\documentclass[11pt,a4paper]{article}

\usepackage{amsmath,amsthm,amsopn,amsfonts,amssymb,dsfont,color}
\usepackage{psfrag}
\usepackage{geometry}
\usepackage{graphicx}
\usepackage[utf8]{inputenc}
\usepackage[english]{babel}
\usepackage[active]{srcltx}

\usepackage{a4,a4wide}
\usepackage{color}
\usepackage{enumerate}

\def\ep{{\varepsilon}}
\def\R{\mathbb R}

\newtheorem{theo}{\textbf{Theorem}}[section]

\newtheorem{lem}[theo]{\textbf{Lemma}}
\newtheorem{prop}[theo]{\textbf{Proposition}}

\newtheorem{cor}[theo]{\textbf{Corollary}}

\newtheorem{assumption}[theo]{\textbf{Assumption}}
\newtheorem{rem}[theo]{\textbf{Remark}}

\title{Interplay of nonlinear diffusion, initial tails and Allee effect on the speed of invasions}
\date{}

\begin{document}

\maketitle

\begin{center}
{\large\bf Matthieu Alfaro \footnote{Institut Montpelli\'erain Alexander Grothendieck, CNRS, Univ. Montpellier, France.\\ E-mail: matthieu.alfaro@umontpellier.fr} and Thomas Giletti \footnote{IECL, Universit\'{e} de Lorraine, B.P. 70239, 54506
Vandoeuvre-l\`{e}s-Nancy Cedex, France.\\
 E-mail:
thomas.giletti@univ-lorraine.fr}.} \\
[2ex]
\end{center}




\vspace{10pt}

\begin{abstract} We focus on the spreading properties of solutions of monostable equations with nonlinear diffusion. We consider both the porous medium diffusion and the fast diffusion regimes. Initial data  may have heavy tails, which tends to accelerate the invasion phenomenon. On the other hand,
 the nonlinearity may  involve a weak Allee effect, which tends to slow down the process. We study the balance between these three effects (nonlinear diffusion, initial tail, KPP nonlinearity/Allee effect), revealing the  separation between \lq\lq no acceleration" and \lq \lq acceleration''. In most of the cases where acceleration occurs, we also give an accurate estimate of the position of the level sets.
\\


\end{abstract}

\maketitle

\section{Introduction} \label{s:intro}

In this paper we are concerned with the {\it spreading properties}
of $u(t,x)$ the solution of the nonlinear monostable reaction-diffusion
equation
\begin{equation}\label{eq}
\partial _t u=\partial _{xx}(u^{m})+f(u),\quad t>0,\, x\in \R.
\end{equation}
We consider both the {\it porous medium diffusion} regime $m>1$ and the {\it fast diffusion} regime $0<m<1$, the {\it linear diffusion} case $m=1$ being already well understood. The typical nonlinearity $f$ we have in mind is $f(s)=rs^{\beta}(1-s)$, with $r>0$ and either $\beta=1$ (Fisher-KPP) or 
$\beta >1$ ({\it Allee effect}).  Equation \eqref{eq} is supplemented with a nonnegative initial data which is {\it front-like} and may have a {\it heavy
tail}, say $u_0(x)\sim \frac{1}{x^{\alpha}}$ for some $\alpha >0$ as $x\to+\infty$.  Our main goal is to understand the interplay between nonlinear diffusion (measured by $m>0$), the initial tail (measured by $\alpha\in(0,+\infty]$) and the behavior of the nonlinearity at 0 (measured by $\beta\geq 1$), and to determine if propagation occurs by accelerating or not. In the former case, we also aim at estimating the
position of the level sets of $u(t,\cdot)$ as $t\to +\infty$, typically
revealing that they travel exponentially or polynomially fast.

\medskip

\noindent{\bf Nonlinear diffusion.} The well-posedness of the Cauchy problem for the degenerate diffusion equation
$$
\partial _t u=\Delta (u^{m}), \quad t>0,\, x\in \R ^{N},
$$
is now well understood as long as $m>1$ or $m_c:=\max(0,1-\frac{2}{N})<m<1$. Precise statements and results can be found in \cite{Her-Pie-85}, \cite{Kal-87}, \cite{book-Wu-et-al}, \cite{Vaz-book}, \cite{Vaz-book2}, \cite{book-Das-Ken}  and the references therein. The main feature of the regime $m>1$ is that the equation degenerates at the points where $u=0$. Hence, the so-called {\it Barenblatt self-similar solutions} exhibit a free boundary, a loss of regularity of solutions occurs and disturbances propagate with finite speed. This is in sharp contrast with the infinite speed of propagation of solutions of the heat equation ($m=1)$ and of the fast diffusion equation ($m<1$). See e.g. \cite{Aud-Vaz-16} and \cite{Aud-Vaz-17}.

{}From the population dynamics point of view, let us briefly explain the role of introducing nonlinear effects into the dispersal behavior of a species. After the observation of arctic ground squirrels migrating from densely populated areas into sparsely populated ones (even if the latter is less favorable) in \cite{Car-71}, Gurney and Nisbet \cite{Gur-Nis-75}, Gurtin and MacCamy \cite{Gur-Mac-77} introduced porous medium diffusion ($m>1$) in models. On the other hand, in order to adapt to low density mate distributions, individuals may need to travel extreme distances, thus leading to accelerating range expansion, for which fast diffusion ($m<1$) may be pertinent: see \cite{Kin-Mac-03} and the references therein for the propagation of early Palaeoindian hunter-gatherers.

\medskip

\noindent{\bf KPP nonlinearities.} In some
population dynamics models, a common assumption is that the growth
is only slowed down by the intra-specific competition, so that the
growth per capita is maximal at small densities. This leads to
consider reaction-diffusion equations with
nonlinearities $f$ of the Fisher-KPP type, namely
$$
f(0)=f(1)=0, \quad\text{ and } \quad 0<f(s)\leq f'(0)s, \quad \forall s\in(0,1).
$$
The simplest example $f(s)=rs(1-s)$, $r>0$, was first introduced by Fisher
\cite{Fis-37} and Kolmogorov, Petrovsky and Piskunov
\cite{Kol-Pet-Pis-37} to model the spreading of advantageous genetic
features in a population.

In such situations, it is well known that the way the front-like
initial data ---~in the sense of Assumption \ref{ass:initial}~---
approaches zero at $+\infty$ is of crucial importance on  the
propagation, that is the invasion for large times of the unstable
steady state $u\equiv 0$ by the stable steady state $u\equiv 1$. 

Let us start with the linear diffusion case $m=1$. For initial data with an exponentially bounded tail (or light tail) at
$+\infty$, there is a {\it spreading speed} $c\geq c^*:=2\sqrt{f'(0)}$
which is selected by the rate of decay of the tail. There is a large literature on such results and improvements, and we only mention the seminal works  \cite{Fis-37},
\cite{Kol-Pet-Pis-37}, \cite{McKea-75}, \cite{Had-Rot-75},
\cite{Kam-76}, \cite{Uch-78}, \cite{Aro-Wei-78}. On the other hand, Hamel and Roques \cite{Ham-Roq-10} recently
considered the case of initial data with heavy (or not
exponentially bounded) tail, namely
\begin{equation*}\label{def:heavy}
\lim_{x\to +\infty} u_0(x)e^{\ep x}=+\infty, \quad \forall \ep >0 .
\end{equation*}
They proved that, for
any $\lambda\in(0,1)$, the $\lambda$-level set of $u(t,\cdot)$
travels infinitely fast as $t\to +\infty$, thus revealing an
acceleration phenomenon. Also, the location of these level
sets is estimated in terms of the decaying rate to zero of the initial data.

Much less is known on the propagation of solutions to \eqref{eq} in the nonlinear diffusion case~$m\neq 1$. Concerning the porous medium regime $m>1$, we mention the works \cite{Kam-Ros-04}, \cite{Aud-Vaz-16} where propagation at constant speed is analyzed. In this paper, we further consider some cases where acceleration occurs because of a heavy initial tail. Concerning the fast diffusion regime $0<m<1$, let us mention the works
\cite{Kin-Mac-03}, \cite{Aud-Vaz-17} where acceleration (already induced by diffusion, whatever the initial tail) is investigated. In this paper, we refine those results and provide precise estimates on the location of the accelerating level sets.

\medskip

\noindent{\bf Allee effect.} In population dynamics, due for instance to
 the difficulty to find mates or to the lack of genetic diversity at low density, the KPP
  assumption is unrealistic in some situations. In other words, the growth per capita is no longer maximal at small densities, which is referred to as an Allee effect.

In this Allee effect context, if $f'(0)>0$ the situation ---~even if more complicated~--- is more or less comparable to the
 KPP situation. On the other hand,
 much less is known in the degenerate situation where $f'(0)=0$, for which typical nonlinearities take the form
$$
f(s)=rs^\beta(1-s), \quad r>0,\, \beta >1.
$$

In the linear diffusion case $m=1$, propagation at constant speed in presence of an Allee effect was for instance studied in \cite{Aro-Wei-78}, \cite{Xin-93}, \cite{Beb-Li-Li-97}, \cite{Zla-05}. Very recently the balance between the strength of the Allee effect (which tends to
slow down the invasion process) and heavy tails (which tend to
accelerate it) was studied in \cite{Alf-tails}. 
 For algebraic tails, the exact
separation between acceleration or not (depending on the strength of the Allee effect) was obtained. Also, when acceleration occurs,  the location of the level sets of the solution was precisely estimated.

To the best of our knowledge there are very few results on the propagation of solutions to~\eqref{eq} combining nonlinear diffusion $m\neq 1$ and an Allee effect (say $\beta >1$). Let us mention the work \cite{Med-03} which, for $m>1$ and Heaviside type initial data, proves propagation at constant speed. In this paper, we prove both \lq \lq no acceleration'' and \lq\lq acceleration'' results, depending on the parameters of diffusion $m>0$, the initial tail $\alpha\in(0,+\infty]$, and the behavior of $f$ at zero  $\beta\geq 1$. Also, when acceleration occurs, we provide sharp estimates of the level sets of the solution.

\begin{rem}[Nonlocal diffusion] Related results exist for the integro-differential
equation of the KPP type
\begin{equation}\label{eq-garnier}
\partial _t u=J*u-u+f(u),
\end{equation}
where the kernel $J$ allows to take into account rare long-distance dispersal events. Here, the
initial data is typically compactly supported and this is the tail of the dispersion kernel $J$ that determines how
fast is the invasion. If the kernel is exponentially bounded, then propagation occurs at a constant speed, as can be
 seen in \cite{Wei-82}, \cite{Cov-preprint}, \cite{Cov-Dup-07} among others.
 More recently, Garnier \cite{Gar-11}  proved an acceleration phenomenon for kernels which are not
  exponentially bounded, so that \eqref{eq-garnier} is an accurate model to explain the Reid's paradox of rapid plant migration (see~\cite{Gar-11} for references on this issue). As far as the integro-differential equation~\eqref{eq-garnier} with an
Allee effect is concerned, we refer to
\cite{Alf-fujita}, \cite{Alf-Cov-preprint} for results on the  balance between the Allee effect and dispersion kernels with heavy tails. 

To conclude on acceleration phenomena in Fisher-KPP type equations, let us mention the case
when the Laplacian is replaced by the generator of a Feller semigroup, a typical example being
\begin{equation}\label{eq-Cab-Roq-13}
\partial _t u=-(-\partial _{xx})^\alpha u+f(u), \quad 0<\alpha<1,
\end{equation}
where $-(-\partial _{xx})^{\alpha}$ stands for the Fractional Laplacian, whose symbol is $\vert\xi\vert ^{2\alpha}$.
 In this context, it was proved by Cabr\'e and Roquejoffre \cite{Cab-Roq-13} that, for a compactly supported
 initial data, acceleration always occurs, due to the algebraic tails of the Fractional Laplacian. Last, notice that the question of acceleration or not in the nonlocal equation \eqref{eq-Cab-Roq-13} with
 an Allee effect has been recently solved by  Gui and Huan \cite{Gui-Hua-preprint}.
\end{rem}

\section{Assumptions and main results}\label{s:results}

Through this work, and even if not recalled, we always make the following assumption
 on the initial condition. 

\begin{assumption}[Initial condition]\label{ass:initial} The initial condition $u_0:\R\to \left[0,1\right]$ is uniformly continuous and asymptotically front-like, in the sense that
\begin{equation*}
\label{front-like}
u_0>0 \; \text{ in } \R,\quad \liminf _{x\to -\infty} u_0(x)>0,\quad \lim _{x\to +\infty} u_0(x)=0.
\end{equation*}
\end{assumption}
As far as the nonlinearity $f$ is concerned, we always assume the following.
\begin{assumption}[Monostable nonlinearity]\label{ass:f} The nonlinearity $f:\left[0,1 \right]\to \R$ is of the class~$C^{1}$, and is of the monostable type, in the sense that
$$
f(0)=f(1)=0, \quad f>0 \; \text{ in } (0,1).
$$
\end{assumption}
Notice that, in each result, we clearly state the decay of the tail of the initial data as well as the behavior of $f(u)$ as $u \to 0$ (Fisher KPP vs. Allee effect), which therefore we did not include in the above assumptions. 
The simplest examples of nonlinearities satisfying Assumption~\ref{ass:f} are
 given by $f(s)=r s^\beta(1-s)$, $r>0$, with $\beta =1$ (Fisher-KPP) or with $\beta >1$ (Allee effect).

In the sequel, we always denote by $u(t,x)$ the solution of \eqref{eq} with initial condition $u_0$. From the above assumptions and the comparison principle, one gets $0\leq u(t,x)\leq 1$ and even
\begin{equation*}\label{strict}
0<u(t,x)<1, \quad \forall (t,x)\in(0,+\infty)\times\R.
\end{equation*}
The strict upper bound can be inferred from the strong maximum principle, but the degenerate diffusion at zero (when $m\neq 1$) prevents such an argument for the strict lower bound, which however follows from \cite[Corollary~4.4]{Vaz-book2} when $m>1$ (recall that $u_0$ is continuous and positive) and from \cite[Theorem~4.6]{Vaz-book2} when $0<m<1$.

Since the initial data is front-like, it is very expected that the state $u\equiv 1$ does invade the whole line $\R$ as $t\to+\infty$: there is $c _0>0$ such that 
\begin{equation}
\label{invasion}
\lim _{t\to +\infty} \inf _{x\leq c _0 t} u(t,x)=1,
\end{equation}
meaning  that propagation is at least linear. We also have 
\begin{equation}
\label{zero-a-droite}
\lim _{x\to+\infty}u(t,x)=0,\quad \forall t\geq 0.
\end{equation}
For the sake of completeness, these preliminary facts \eqref{invasion} and \eqref{zero-a-droite} will be proved in Section \ref{s:prelim}.

In order to state our results we define, for any $\lambda \in (0,1)$ and $t\geq 0$,
$$
E_\lambda(t):=\{x\in\R:\, u(t,x)=\lambda\}
$$
the $\lambda$-level set of $u(t,\cdot)$. In view of \eqref{invasion} and \eqref{zero-a-droite}, for any $\lambda\in(0,1)$, there is a time $t_\lambda >0$ such that
\begin{equation}
\label{nonvide}
\emptyset \neq E_\lambda(t) \subset (c _0 t, +\infty), \quad \forall t\geq t_\lambda.
\end{equation}

Our first main result is concerned with the acceleration phenomenon that occurs for any~$m>0$ as soon as the nonlinearity is of the Fisher-KPP type. For linear diffusion $m=1$ this was proved in \cite{Ham-Roq-10}. We extend the result to porous medium diffusion $m>1$ and fast diffusion $0<m<1$.

\begin{theo}[Acceleration in the Fisher-KPP case]\label{th:FKPP_porous_FDE}
Let $m > 0$ and $\alpha >0$ be given. Assume that there are $C>0$, $\overline C>0$ and $x_0>1$ such that
\begin{equation}
\label{algebraic-acc-kpp}
\frac{\overline C}{x^{\alpha}}\geq u_0(x)\geq \frac{C}{x^{\alpha}},\quad \forall x\geq x_0,
\end{equation}
as well as $r>0$ and $s_0\in(0,1)$ such that
\begin{equation}
\label{nonlinearity-acc-kpp}
f(s)\geq r s,\quad \forall 0\leq s\leq s_0.
\end{equation}
Select $\overline r >0$ such that
\begin{equation}
\label{nonlinearity-accbis-kpp}
\overline r s\geq f(s),\quad \forall 0\leq s\leq 1.
\end{equation}
Then, for any $\lambda \in(0,1)$, any small $\ep>0$,  there is a
time $T_{\lambda,\ep}\geq t_\lambda$ such that
\begin{equation}
\label{levelset-kpp}
E_\lambda(t)\subset (x^-(t),x^{+}(t)),\quad \forall t\geq  T _{\lambda,\ep},
\end{equation}
where
$$
 x^-(t):=e^{(r-\ep)\Gamma t}, \quad x^+(t):=e^{(\overline r+\ep)\Gamma t}, \quad \Gamma :=\max\left(\frac{1-m}2,\frac{1}{\alpha}\right).
 $$
\end{theo}

The above result indicates that the level sets of the solution travel exponentially fast.
For any $m>1$, we have $\Gamma=\frac 1 \alpha$ (independent on $m$) for all $\alpha >0$ and the estimate is the same as that of \cite{Ham-Roq-10} when $m=1$. On the other hand if $0<m<1-\frac 2 \alpha$, we have $\Gamma =\frac{1-m}{2}$ and, due to fast diffusion, the estimate is in contrast with that of \cite{Ham-Roq-10}.

When $m\geq 1$ the heavy tail assumption, i.e. the lower bound in \eqref{algebraic-acc-kpp}, is crucial for the acceleration results in Theorem \ref{th:FKPP_porous_FDE}
to hold. On the other hand, when $0<m<1$ this is not necessary since the fast diffusion equation makes the tail of the solution (at least) algebraically heavy at any positive time \cite{Her-Pie-85}, see subsection \ref{s:heavy} for details. To shed light on this phenomenon, we state the following corollary which is a rather straightforward consequence of Theorem \ref{th:FKPP_porous_FDE} and subsection \ref{s:heavy}.

\begin{cor}[Acceleration in the Fisher-KPP case, $0<m<1$]\label{cor:FKPP_porous_FDE}
Let $0<m <1$ be given. Assume that there are $\overline C>0$ and $x_0>1$ such that
\begin{equation}
\label{algebraic-acc-kpp-cor}
\frac{\overline C}{x^{\frac{2}{1-m}}}\geq u_0(x),\quad \forall x\geq x_0,
\end{equation}
as well as $r>0$  and $s_0\in(0,1)$ such that \eqref{nonlinearity-acc-kpp} holds. Select  $\overline r >0$ as in \eqref{nonlinearity-accbis-kpp}. Then the conclusions of Theorem \ref{th:FKPP_porous_FDE} are valid with $\Gamma=\frac{1-m}2$.
\end{cor}

{}From now on, we consider an Allee effect by letting $f(s)$ behave like $s^{\beta}$,  $\beta >1$, as $s\to 0$. Our next theorem shows that the acceleration phenomenon disappears when $\beta$ is large enough.

\begin{theo}[No acceleration regime]\label{th:no-acc} Let $m>0$, $\alpha >0$ and $\beta >1$ be such that
\begin{equation}\label{alpha-beta-no-acc}
 \beta \geq \max\left(1+\frac 1 \alpha,2-m\right).
\end{equation}
Assume that there are $\overline C>0$ and $x_0>1$ such that
\begin{equation}
\label{algebraic-no-acc}
u_0(x)\leq \frac{\overline C}{x^\alpha},\quad \forall x\geq x_0 ,
\end{equation}
as well as $\overline r>0$ and $s_0\in(0,1)$ such that
\begin{equation}
\label{nonlinearity-no-acc}
 f(s)\leq \overline r s^\beta,\quad \forall 0\leq s\leq s_0.
\end{equation}
Then, there is a speed $c>0$ such that, for any $\lambda \in(0,1)$,
there is a time $T_\lambda \geq t_\lambda$ such that
\begin{equation}
\label{no-acc}
\emptyset \neq E_\lambda(t) \subset (c _0 t, ct), \quad \forall t\geq T_\lambda.
\end{equation}
\end{theo}

Next, we go back to the acceleration regime and look at the intermediate values of $\beta$. We need to distinguish the $m> 1$ porous medium regime from the $0<m<1$ fast diffusion regime. The former is sharply solved by the following, which indicates strong similarities with the linear diffusion regime $m=1$ studied in \cite{Alf-tails}.
\begin{theo}[Localization of the accelerating level sets, $m> 1$]\label{th:acc} Let $m> 1$, $\alpha >0$ and~$\beta >1$ be such that
\begin{equation*}\label{alpha-beta-acc}
 \beta < 1+\frac 1 \alpha.
\end{equation*}
Assume that there are $C>0$, $\overline C>0$ and $x_0>1$ such that
\begin{equation*}
\label{algebraic-acc}
\frac{\overline C}{x^{\alpha}}\geq u_0(x)\geq \frac{C}{x^\alpha},\quad \forall x\geq x_0 ,
\end{equation*}
as well as $r>0$, $\overline r>0$, and $s_0\in(0,1)$ such that
\begin{equation}
\label{nonlinearity-acc}
f(s)\geq r s^\beta  ,\quad \forall 0\leq s\leq s_0,
\end{equation}
and
\begin{equation}
\label{nonlinearity-accbis}
\overline r s^\beta\geq f(s),\quad \forall 0\leq s\leq 1.
\end{equation}
Then, for any $\lambda \in(0,1)$, any small $\ep>0$,  there is a
time $T_{\lambda,\ep}\geq t_\lambda$ such that
\begin{equation}
\label{levelset}
E_\lambda(t)\subset (x^-(t),x^{+}(t)),\quad \forall t\geq  T _{\lambda,\ep},
\end{equation}
where
$$
 x^-(t):=\left((r-\ep)C^{\beta -1}(\beta -1)t\right)^{\frac{1}{\alpha(\beta-1)}}, \quad x^+(t):=\left((\overline r+\ep)\overline C ^{\beta -1}(\beta -1)t\right)^{\frac{1}{\alpha(\beta-1)}}.
 $$
\end{theo}

The similarity with the linear diffusion regime can be understood from the fact that, when~$m>1$, diffusion is slower at low values of $u$ and therefore it is not expected to play the driving role in the acceleration phenomenon. The picture is now completely understood in the case $m>1$ and is summarized in Figure \ref{fig:pme}.

\begin{figure}[!h]
\begin{center}
\includegraphics[scale=1.6]{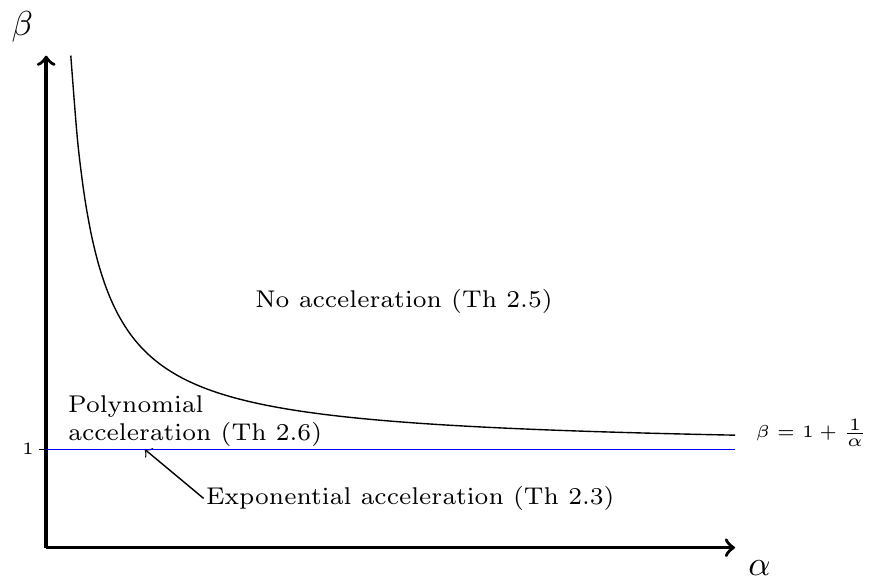}
\caption{Summary of our results in the porous medium diffusion case, $m>1$.} \label{fig:pme}
\end{center}
\end{figure}

On the other hand, the situation is different when $m \in (0,1)$, as the fast diffusion may then overcome the reactive growth. Let us thus turn to the situation when $m\in (0,1)$ and $\beta$ is  in the intermediate range where we expect acceleration. 

\begin{theo}[Localization of the accelerating level sets, $0<m <1$]\label{th:acc-FDE} Let $m \in (0, 1)$, $\alpha \in (0, \frac{2}{1-m})$ and $\beta>1$ be such that
\begin{equation}\label{alpha-beta-acc-FDE}
 \beta <  \min \left( 1+\frac 1 \alpha , m + \frac 2 \alpha \right).
\end{equation}
Assume that there are $C>0$, $\overline C>0$ and $x_0>1$ such that
\begin{equation}
\label{algebraic-acc-FDE}
\frac{\overline C}{x^{\alpha}}\geq u_0(x)\geq \frac{C}{x^\alpha},\quad \forall x\geq x_0,
\end{equation}
as well as $r>0$, $\overline r>0$, and $s_0\in(0,1)$ such that
\begin{equation}
\label{nonlinearity-acc-FDE}
f(s)\geq r s^\beta  ,\quad \forall 0\leq s\leq s_0,
\end{equation}
and
\begin{equation}
\label{nonlinearity-accbis-FDE}
\overline r s^\beta\geq f(s),\quad \forall 0\leq s\leq 1.
\end{equation}
Then, for any $\lambda \in(0,1)$, any small $\ep>0$,  there is a
time $T_{\lambda,\ep}\geq t_\lambda$ such that
\begin{equation}
\label{levelset-FDE}
E_\lambda(t)\subset (x^-(t),x^{+}(t)),\quad \forall t\geq  T _{\lambda,\ep},
\end{equation}
where
$$
 x^-(t):=\left((r-\ep)C^{\beta -1}(\beta -1)t\right)^{\frac{1}{\alpha(\beta-1)}}, \quad x^+(t):=\left((\overline r+\ep)\overline C ^{\beta -1}(\beta -1)t\right)^{\frac{1}{\alpha(\beta-1)}}.
 $$
\end{theo}
Let us comment the assumptions on the parameters $\alpha$ and $\beta$ of the above theorem. First of all, as already mentioned above and as will be detailed in subsection \ref{s:heavy}, fast diffusion increases the tail of solutions so that the range $\alpha > \frac{2}{1-m}$ is, in some sense, irrelevant. As a consequence the  assumption $\alpha\in(0,\frac{2}{1-m})$ actually only rules out the critical case $\alpha = \frac{2}{1-m}$. 

Moreover, \eqref{alpha-beta-acc-FDE} means that $\beta$ lies below two hyperbolae. The first one already appeared in Theorem \ref{th:no-acc}. The second one, which is relevant only in the regime $\frac{1}{1-m}<\alpha\leq \frac{2}{1-m}$, seems to appear for technical reasons. Indeed in the region $\frac{1}{1-m}<\alpha\leq \frac{2}{1-m}$ and $m+\frac{2}{\alpha}\leq \beta <2-m$ (which is  covered neither by Theorem \ref{th:no-acc} nor by Theorem \ref{th:acc-FDE}), even if we cannot precisely localize the position of the level sets, we can still prove acceleration. 

\begin{theo}[Acceleration in the remaining \lq\lq parameters region'', and even more]\label{th:reste}
Let $m \in (0,1)$ and $1 < \beta < 2-m$ be given. Assume that there are  $r>0$ and $s_0\in(0,1)$ such that~\eqref{nonlinearity-acc-FDE} holds.
\begin{itemize}
\item[$(i)$] Then, for any $c>0$, we have
\begin{equation*}
\lim _{t\to +\infty} \inf _{x\leq c t} u(t,x)=1.
\end{equation*}
\item[$(ii)$] If furthermore \eqref{algebraic-acc-FDE} holds with $\frac{1}{1-m} < \alpha \leq \frac{2}{1-m}$ and $m + \frac{2}{\alpha} \leq \beta < 1 +\frac{1}{\alpha},$ then for any~$\lambda \in (0,1)$, any small $\varepsilon >0$, there is a time $T_{\lambda,\varepsilon} \geq t_\lambda$ such that
\begin{equation}\label{dernier-truc}
E_\lambda (t) \subset (x^- (t), +\infty), \quad \forall t \geq T_{\lambda,\varepsilon},
\end{equation}
where
$$
 x^-(t):=\left((r-\ep)C^{\beta -1}(\beta -1)t\right)^{\frac{1}{\alpha(\beta-1)}}.
$$
If instead \eqref{algebraic-acc-kpp-cor} holds, then the same conclusion follows when replacing $\alpha$ by $\frac{2}{1-m}$ in the definition of $x^- (t)$.
\end{itemize}
\end{theo}

The first part of Theorem~\ref{th:reste} shows that the speed of propagation is infinite when $1 < \beta < 2 -m$ (we already treated the case $\beta =1$), requiring only Assumption \ref{ass:initial} on the initial datum. In other words, when the Allee effect is small w.r.t. fast diffusion, acceleration occurs regardless of the initial tail. 

The second part of Theorem \ref{th:reste} deals with the situation when $m + \frac{2}{\alpha} \leq \beta < 1 + \frac{1}{\alpha}$. Indeed, in this situation and when the diffusion is of the porous medium type, we have shown in Theorem~\ref{th:acc} that polynomial acceleration occurs. However, in the fast diffusion case, this parameter range was missing from Theorem~\ref{th:acc-FDE}. The reason is that we are not able to accurately locate the level sets. Still, Theorem~\ref{th:reste} shows that the acceleration phenomenon when $0 < m<1$ is at least as strong as in the case $m \geq 1$, as one may have expected.

\begin{rem}\label{rem:upper} In the regime $0<m<1$, $\beta >1$ and $m+\frac 2 \alpha \leq \beta <2-m$ (see Figure \ref{fig:fde}), assuming \eqref{nonlinearity-accbis-FDE} and the upper bound in \eqref{algebraic-acc-FDE}, that is $u_0(x)\leq \frac{\overline C}{x^\alpha}$, a monomial type upper bound on the level sets can be obtained. Indeed, we can use a comparison argument with the case where  $\alpha$ is replaced by a smaller $\alpha _0$ which falls into the results of Theorem \ref{th:acc-FDE}. Hence, for any small $\ep>0$, we let $\alpha _0:=\frac{2}{\beta -m+\ep}$ and 
deduce from Theorem \ref{th:acc-FDE} the upper bound
$$
x^+(t):=\left((\overline r+\ep)\overline C ^{\beta -1}(\beta -1)t\right)^{\frac{\beta-m+\ep}{2(\beta-1)}}.
$$
For the sake of conciseness and because it is a simple consequence of our previous theorems, we omitted this upper bound  in Theorem \ref{th:reste}.
\end{rem}

Using the fact that \lq\lq fast diffusion increases the tail  of solutions'' (see subsection \ref{s:heavy}), we can summarize our results for the case $m\in(0,1)$ in Figure \ref{fig:fde}. Notice that the position of the level sets is not yet completely understood in the regions covered by Theorem~\ref{th:reste}. Indeed, in the region covered by Theorem~\ref{th:reste}~$(ii)$ we are equipped with lower and upper monomial type estimates that typically differ (see Remark \ref{rem:upper}), whereas in the region covered by Theorem~\ref{th:reste}~$(i)$ a monomial type lower bound is not even available. Such difficulties  previously arose in a integro-differential equation \cite{Alf-Cov-preprint}, where the balance between the Allee effect and the tails of the dispersal kernel was studied.

\begin{figure}[!h]
\begin{center}
\includegraphics[scale=1.6]{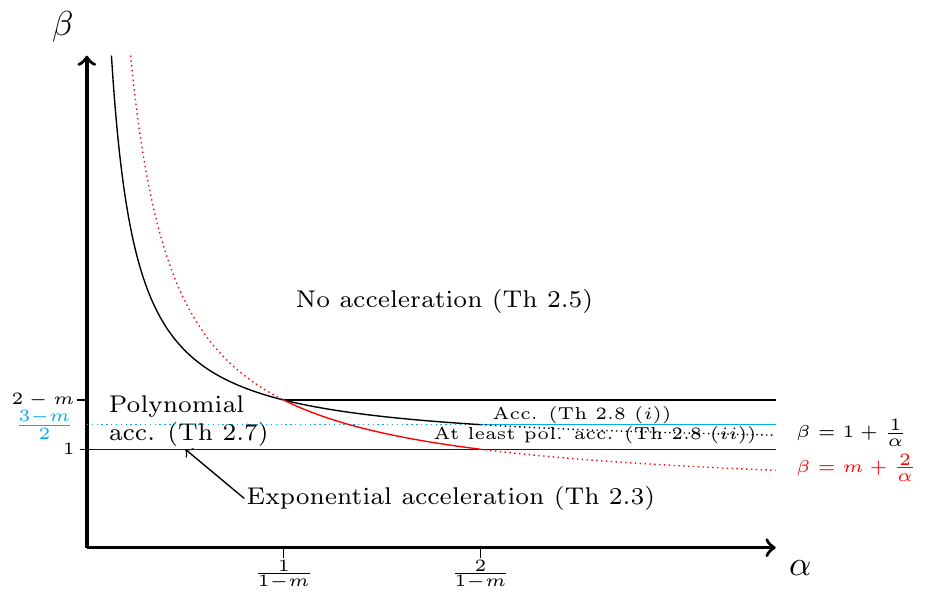}
\caption{Summary of our results in the fast diffusion case, $m\in(0,1)$. The different parameter regions are delimited by the solid lines.} \label{fig:fde}
\end{center}
\end{figure}

\medskip

The paper is organized as follows. In Section \ref{s:prelim}, we prove the preliminary results \eqref{invasion} and~\eqref{zero-a-droite} together with the statement $(i)$ of Theorem~\ref{th:reste}. In Section \ref{s:no-acceleration}, we consider the regime where there is no acceleration, that is we prove Theorem \ref{th:no-acc}. Next, in Section~\ref{s:PME}, we focus on the porous medium regime: we prove simultaneously Theorem~\ref{th:FKPP_porous_FDE} (case $m>1$) and Theorem~\ref{th:acc}. Let us point out that, while we stated those results separately because of the different conclusions (exponential vs. algebraic estimates), the proofs will be almost identical. Then, in Section~\ref{s:FDE-kpp}, we turn to the fast diffusion regime: we complete the proof of Theorem~\ref{th:FKPP_porous_FDE} (case $0<m<1$), and prove Theorem~\ref{th:acc-FDE}. Last, in Appendix \ref{s:appendix}, we complete the proof of Theorem \ref{th:reste} by dealing with statement $(ii)$.

\section{Positive and infinite speeds of propagation}\label{s:prelim}

In this section, we show propositions which provide subsolutions with a support bounded from above and moving with a constant speed $c >0$, depending on the parameters of the equation~\eqref{eq}. This enables us not only to prove statement~$(i)$ of Theorem \ref{th:reste} (infinite speed of propagation when $m\in(0,1)$, $1 \leq \beta < 2-m$), but also to show that level sets always move to the right at least linearly, namely estimate \eqref{invasion} which will be used  several times throughout this paper. After that, we also prove the preliminary result \eqref{zero-a-droite}.

The subsequent propositions all rely on a similar argument. First, since we will look at 
traveling front type solutions or subsolutions (i.e. whose shape remains constant in time in an appropriate moving frame), the equation~\eqref{eq} reduces to a nonlinear ordinary differential equation. By a change of variables introduced by Engler \cite{Eng-85}, see also \cite{Gil-Ker-04},  we are able to further reduce the problem to a situation where diffusion is linear and where we use a phase plane analysis.

\medskip

We start with the case where $g(s):=mf(s)s^{m-1}$ has \lq\lq infinite slope at zero'', meaning that $\lim _{s\to 0} \frac{g(s)}{s}=+\infty$. This covers the case $m\in(0,1)$ and $1< \beta <2-m$ of Theorem~\ref{th:reste}.

\begin{prop}[Material for subsolutions traveling at any speed $c>0$]\label{prop:fin1}
Let $m \in (0,1)$  be given. Assume that $g(s):=mf(s)s^{m-1}$ is such that $\lim _{s\to 0} \frac{g(s)}{s}=+\infty$. Then for any $\delta \in (0,1)$, any $c >0$, there are $x_c >0$ and a decreasing function $U_c  : [0, x_c]\to [0,\delta] $ which solves
\begin{equation}\label{eq:ode_nl}
 (U_c^m)'' + c U_c ' + f(U_c) =0 \quad \text{ on } (0,x_c),\end{equation}
as well as the boundary conditions
\begin{equation}\label{eq:ode_boundary}
U_c (0) = \delta, \qquad U_c ' (0) = 0 , \qquad U_c (x_c ) = 0.
\end{equation}
\end{prop}

\begin{proof} Let $\delta\in(0,1)$ and $c >0$ be given. We first introduce the equation
\begin{equation}\label{eq:chv}
V'' + c V' + g(V)= 0,
\end{equation}
where $g(s):= m f(s) s^{m-1}$ is of the class $C^{1}$ on $(0,1]$ but not on $[0,1]$. Denote by $V_c(y)$ its solution starting from $V_c(0)=\delta$, $V_c '(0)=0$. By a phase plane analysis, one infers that one of the three following statements holds:
\begin{enumerate}[$(i)$]
\item 
  $V_c$ remains positive on $(0,+\infty)$, $V_c'$ remains negative on $(0,+\infty)$, $(V_c(y),V_c '(y))\to (0,0)$ as $y\to +\infty$, and there exists a sequence $y_k \to +\infty$ such that
\begin{equation}\label{sous-isocline}
  c V_c'(y_k)+g(V_c (y_k))\leq 0. 
\end{equation}
\item there is $y_c>0$ such that $V_c$ remains positive on $(0,y_c)$, $V_c'$ remains negative on $(0,y_c)$, $(V_c(y_c),V_c '(y_c))= (0,0)$, and there exists an increasing sequence $y_k \to y_c$ such that
\begin{equation*}\label{sous-isocline-bis}
  c V_c'(y_k)+g(V_c (y_k))\leq 0. 
\end{equation*}
\item  there is $y_c>0$ such that $V_c (y_c)=0$ and $V_c '<0$ on $(0,y_c]$.
\end{enumerate}
Note that Cauchy-Lipschitz theorem does not apply here, so that case $(ii)$ cannot be immediately excluded. We will now show that $(iii)$ necessarily occurs.

Let us assume $(i)$ and derive a contradiction. Integrating the inequality $-c V_c'\geq V_c ''$ from~$y$ to~$+\infty$ we get
$c V_c (y)\geq -V_c'(y)$. Hence, in view of \eqref{sous-isocline}, we obtain
\begin{equation*}
\label{absurde}
g(V_c(y_k))\leq c ^{2}V_c(y_k)\quad \text{for $k$ large enough},
\end{equation*}
which contradicts the fact that $\lim_{s \to 0 } \frac{g(s)}{s} = +\infty $. In a similar fashion, in the case $(ii)$, integrating the inequality $-c V_c ' \geq V_c ''$ between $y$ and $y_c$ leads to $c V_c (y) \geq - V'_c (y)$. This again results in a contradiction, and we conclude that the solution $V_c$ satisfies~$(iii)$.

Now define
\begin{equation}\label{chgt-var}
U_c (x) : = V_c (\varphi^{-1} (x)) ,
\end{equation}
where
$$\varphi (y) :=  \int_0^y m V_c^{m-1} (s) ds .$$
Note that $\varphi$ is indeed a bijection between $[0,y_c]$ and $[0,x_c]$, with
$$x_c := \int_0^{y_c} m V_c^{m-1} (s) ds.$$
Here we used the fact $V'_c (y_c) < 0$ and $m-1 \in (-1,0)$, hence $V_c^{m-1}$ is integrable on $(0,y_c)$.

Then one can compute that \eqref{eq:chv} rewrites as 
$$m^2 U_c^{2m-2} (\varphi (y)) U_c '' (\varphi (y)) + m^2 (m-1) U_c^{2m-3} (\varphi (y)) ( U_c ' (\varphi (y)))^2 $$
$$+ m c U_c^{m-1} (\varphi (y))  U_c ' (\varphi (y)) + m f(U(\varphi (y)) U(\varphi (y))^{m-1} = 0,$$
thus
$$m U_c^{m-1} U_c '' + m (m-1)  U_c^{m-2} (U_c ' )^2 + c U_c ' + f(U_c) = 0,$$
which is exactly \eqref{eq:ode_nl}. The monotonicity and the boundary conditions \eqref{eq:ode_boundary} are straightforward, and the proposition is proved. \end{proof}

We are now in the position to prove Theorem \ref{th:reste} $(i)$, whose assumptions $m\in(0,1)$, $1< \beta <2-m$ and \eqref{nonlinearity-acc-FDE} allow to use the above proposition.

\begin{proof}[Proof of Theorem \ref{th:reste} $(i)$]
Let us fix $c >0$ and choose $\delta < \liminf_{x\to -\infty} u_0 (x)$. We now extend~$U_c$ by $\delta$ when $x \leq 0$, and by 0 when $x \geq x_c$. For convenience, we still denote the resulting function from $\mathbb{R}$ to $\mathbb{R}$ by $U_c$. We also find some $x_0 \in \mathbb{R}$ so that $U_c (\cdot + x_0) \leq u_0 (\cdot)$. From Proposition~\ref{prop:fin1}, one can check that $U_c (x - c t + x_0)$ is a subsolution of \eqref{eq} in the domain $\{(t,x): t>0,\, x<x_c  -x_0+c t\}$. Recalling that $u(t,x)>0$ for all $t>0$, $x\in \R$, we deduce from the maximum principle that $U_c (x - c t +x_0)\leq u(t,x)$ for all $t>0$, $x\in \R$. We thus infer that 
$$\liminf_{t \to +\infty} \inf_{x \leq c ' t} u(t,x)  \geq \delta ,$$
for any $c ' < c$.

Let us now proceed by contradiction and assume that there exist sequences $t_n \to +\infty$ and~$x_n$ such that $x_n \leq c ' t_n$ and
$$\lim_{n \to +\infty} u(t_n,x_n) < 1.$$
By standard parabolic estimates (thanks to the lower bound on $u$ above, the degeneracy at~0 raises no issue here), we find that $u(t+t_n,x+x_n)$ converges locally uniformly to an entire (i.e. for all $t \in \mathbb{R}$) solution $u_\infty \geq \delta$ of \eqref{eq}. Since $f$ is of the monostable type, it follows that~$u_\infty \geq 1$, which contradicts our choice of $t_n$ and $x_n$. 

Since $c$, hence $c '$, could be chosen arbitrarily large, this proves that the solution propagates with infinite speed when $m \in (0,1)$ and $1 < \beta < 2-m$.
\end{proof}

For the next result, we only require Assumption  \ref{ass:f} concerning the nonlinearity $f$.

\begin{prop}[Material for subsolutions traveling at some speed $c _0>0$]\label{prop:fin2}
Let $m>0$ be given. Then for any $\delta  \in (0,1)$, there exists $c _0 >0$ small enough such that the conclusions of Proposition \ref{prop:fin1} hold.
\end{prop}

\begin{proof}Let $\delta \in (0,1)$ be given. Since $g(s):=mf(s)s^{m-1}$ may be non Lipschitz continuous, we first consider a smaller Lipschitz continuous ignition type nonlinearity. Precisely, we consider a Lipschitz continuous function $\tilde g$ such that $0\leq \tilde g \leq g$, $\tilde g=0$ on $[0,\frac \delta 2]\cup \{1\}$, $\tilde g>0$ on $(\frac \delta 2,1)$. The extensions of $g$, $\tilde g$ by zero outside $[0,1]$ are still denoted by $g$, $\tilde g$. From a phase plane analysis we see that $\tilde V_0(y)$ the solution of the Cauchy problem
$$
V''+\tilde g (V)=0, \quad V(0)=\delta, \quad V'(0)=0,
$$
is global, and that there is $\tilde y_0>0$ such that $\tilde V_0(\tilde y_0)=0$ and $\tilde V_0'<0$ on $(0,\tilde y_0]$. From the continuous dependance of solutions to the Cauchy problem 
\begin{equation}\label{cauchy-tilde}
V''+cV'+\tilde g (V)=0, \quad V(0)=\delta, \quad V'(0)=0,
\end{equation}
w.r.t. parameter $c$, we infer that for $c _0 >0$ small enough the solution satisfies statement $(iii)$ of the proof of Proposition \ref{prop:fin1}. Since $\tilde g\leq g$, we deduce by comparison (in the phase plane, trajectories associated with $g$ are below those associated with $\tilde g$) that the same conclusion  holds for the Cauchy problem \eqref{cauchy-tilde} with $g$ in place of $\tilde g$. Next we define \eqref{chgt-var} and the end of the  proof is similar to that of Proposition \ref{prop:fin1}. 
\end{proof}

We are now in the position to prove \eqref{invasion}.

\begin{proof}[Proof of \eqref{invasion}] The proof uses the above proposition and proceeds as that of Theorem \ref{th:reste}. Let us again notice that our assumption  $u_0 >0$ implies $u(t,x)>0$ for all $t>0$, $x\in \R$ so that the degeneracy/singularity at 0 (depending on the diffusion regime) raises no issue and we can apply the maximum principle on the moving truncated domain $\{(t,x): t>0,\, x<x_{c _0}  -x_0+c_0 t\}$. Details are omitted.

We conclude that, for any $m >0$, there exists $c _0>0$ small enough such that
$$\lim_{t \to +\infty} \inf_{x \leq c _0 t } u(t,x) = 1,$$
as we announced in the introduction.
\end{proof}

To conclude this section, we prove \eqref{zero-a-droite}.

\begin{proof}
[Proof of \eqref{zero-a-droite}] {}From Assumption \ref{ass:f}, there is $\overline r>0$ such that \eqref{nonlinearity-accbis-kpp} holds. By comparison, it is thus enough to consider the solution of
\begin{equation*}
\label{pardessus}
\partial _t u=\partial_{xx} (u^{m})+\overline r u, \quad u(0,\cdot)=u_0.
\end{equation*}
Next, observe that letting (see \cite{Aud-Vaz-16, Aud-Vaz-17})
\begin{equation}\label{vtau}
v(\tau,x):=e^{-\overline r t}u(t,x),
\end{equation}
with
$$
\tau(t):=\begin{cases} \frac 1{(1-m)\overline r}\left(1-e^{-(1-m)\overline r t}\right) &\text{
if } 0<m<1
\\
t  &\text{ if } m=1\\
  \frac 1{(m-1)\overline r}\left( e^{(m-1)\overline r t}-1\right)  &\text{ if  } m>1,
\end{cases}
$$
we see that $v(\tau,x)$ solves ---on the time interval $(0,\tau_\infty)$ with $\tau_\infty=+\infty$ if $m\geq 1$, $\tau_\infty=\frac{1}{(1-m)\overline r}$ if $0<m<1$--- 
\begin{equation}
\label{eq-w}
\partial _\tau v=\partial _{xx} (v^{m}), \quad v(0,\cdot)=u_0.
\end{equation}

Now consider $\ep\in(0,1)$. Select $x_0$ such that $u_0(x)\leq\ep$ for all $x\geq x_0$.  In order to construct a supersolution to \eqref{eq-w}, define 
$$
w(\tau,x):=\min\left(1,\ep+e^{-\mu(x-x_0-\tau)}\right),
$$
where $\mu>0$ is to be selected. Clearly $u_0\leq w(0,\cdot)$. For the points $(\tau,x)$ such that $w(\tau,x)<1$, we have that 
$$
\partial _\tau w-\partial _{xx}(w^m)=\mu e^{-\mu(x-x_0-\tau)}-\mu ^2 m e^{-\mu(x-x_0-\tau)}w^{m-1}-\mu ^{2}m(m-1)e^{-2\mu(x-x_0-\tau)}w^{m-2}
$$
is positive if $\mu>0$ is sufficiently small (recall $\ep\leq w <1$). Since $v(\tau,x)<1$, we can apply a comparison principle and deduce that $v(\tau,x)\leq w(\tau,x)$ for all $\tau\in(0,\tau _\infty)$, $x\in\R$. Hence, for a given $\tau _0\in(0,\tau _\infty)$, we have $\limsup _{x\to +\infty} v(\tau _0,x) \leq \ep$. 
Since $\ep\in(0,1)$ could be chosen arbitrarily small, we get $\lim _{x\to +\infty} v(\tau _0,x)=0$, which in view of \eqref{vtau}, yields \eqref{zero-a-droite}.
\end{proof}

\section{No acceleration regime}\label{s:no-acceleration}

In this short section, we prove Theorem \ref{th:no-acc}. The formal argument is very simple: in order to
avoid acceleration, we aim at finding a speed $c>0$ and a power $p>0$ such that $w(z):=\frac 1{z^p}$ is a supersolution of the associated traveling wave equation for $z>>1$, that is
$$
(w^{m})''(z)+cw'(z)+f(w(z))\leq 0.
$$
In view of \eqref{nonlinearity-no-acc} this is enough to have
$$
\frac{mp(mp+1)}{z^{mp+2}}-\frac{cp}{z^{p+1}}+\frac{\overline r}{z^{p\beta}}\leq 0 \quad \text{ for large } z>>1,
$$
which requires $p+1\leq p\beta$ and $p+1\leq mp+2$. The former condition is never satisfied if~$\beta =1$, and we recast it $\frac{1}{\beta-1}\leq p$; the latter condition is always true when $m\geq 1$ but reduces to $p\leq \frac{1}{1-m}$ when $m<1$. 

On the other
 hand we also need the ordering at initial time, which in view of \eqref{algebraic-no-acc}, requires
 $p\leq \alpha$. Putting these conditions together  one needs $\beta \geq \max\left(1+\frac 1\alpha,2-m\right)$, so that the  condition  \eqref{alpha-beta-no-acc} arises naturally. Let us now make this formal argument precise.

\medskip

We define
$$
p:=\frac{1}{\beta-1},\quad w(z):=\frac{K}{z^p} \quad \text{ for } z\geq z_0:=K^{1/p},
$$
where $K>1$.

\begin{lem}[Supersolutions traveling at constant speed]\label{lem:sursoltw} Let assumptions \eqref{alpha-beta-no-acc} and \eqref{nonlinearity-no-acc} of Theorem \ref{th:no-acc} hold. Then, for any $K>1$, there is $c>0$ such that
$$
(w^{m})''(z)+cw'(z)+f(w(z))\leq 0, \quad \forall z\geq z_0.
$$
\end{lem}

\begin{proof} In view of \eqref{nonlinearity-no-acc}, if $z\geq z_1:=\left(\frac{K}{s_0}\right)^{1/p}>z_0$ then $w(z)=\frac{K}{z^p}\leq s_0$ so that
\begin{eqnarray*}
(w^{m})''(z)+cw'(z)+f(w(z))&\leq & \frac{K^{m}mp(mp+1)}{z^{mp+2}}-\frac{cKp}{z^{p+1}}+\frac{\overline rK^\beta}{z^{p\beta}}\\
&=&\frac{K^{m}mp(mp+1)}{z^{mp+2}}-\frac{cKp-\overline rK^{\beta}}{z^{p+1}}.
\end{eqnarray*}
Recalling that $p+1\leq mp +2$ and choosing $c>\overline r(\beta -1)K^{\beta-1}$, the above is clearly negative
for $z$ large enough, say $z\geq z_2$. Last, on the remaining
compact region $z_0\leq z\leq z_2$, we have
\begin{eqnarray*}
(w^{m})''(z)+cw'(z)+f(w(z))&=&\frac{K^{m}mp(mp+1)}{z^{mp+2}}-\frac{cKp}{z^{p+1}}+f(w(z))\\
&\leq& \frac{K^{m}mp(mp+1)}{z_0 ^{mp+2}}-\frac{cKp}{z_2^{p+1}}+\Vert f\Vert _{L^{\infty}(0,1)}\\
&\leq & 0
\end{eqnarray*}
by enlarging $c$ if necessary.
\end{proof}

\begin{proof}[Proof of Theorem \ref{th:no-acc}] We select $K=\max(1,\overline C)$, where $\overline C>0$ is
  the constant that appears in \eqref{algebraic-no-acc}, and $c>0$ the associated speed given by the above lemma. We then introduce
$$
v(t,x):=\min\left(1,w(x-x_0+1-ct)\right),
$$
so that 
$$
v(0,x)=\min\left(1,\frac{K}{(x-x_0+1)^p}\right)\geq u_0(x),
$$
in view of $u_0\leq 1$, the assumption on the tail \eqref{algebraic-no-acc}, $K\geq \overline C$ and $p=\frac{1}{\beta-1}\leq \alpha$.

Recall that $u < 1$ for positive times. Also from the above lemma we have $(\partial _t v-\partial _{xx}(v^{m})-f(v))(t,x)=(-cw'-(w^{m})''-f(w))(z) \geq 0 $ in the region where $v(t,x)<1$, that is $z:= x - x_0 +1 - ct > z_0$. In other words, $v$ is a (generalized) supersolution of equation~\eqref{eq}, and applying a comparison principle on a right half-domain we get that
$$
u(t,x)\leq v(t,x)=\min\left(1,w(x-x_0+1-ct)\right).
$$
Now, let $\lambda\in(0,1)$ be given. In view of \eqref{nonvide}, for $t\geq t_\lambda$,  we can pick $x\in E_\lambda(t)$, and the above inequality enforces
$$
x\leq x_0-1+\left(\frac{K}{\lambda}\right)^{\beta-1}+ct\leq (c+1)t,
$$
for all $t\geq T_\lambda$, if $T_\lambda \geq t_\lambda$ is sufficiently large. This proves the
 upper bound in \eqref{no-acc}. The lower bound in \eqref{no-acc} being known since \eqref{nonvide}, this completes the proof of Theorem \ref{th:no-acc}.
 \end{proof}
 
 \section{Acceleration regime for porous medium diffusion}\label{s:PME}

In this section, we prove both Theorem~\ref{th:FKPP_porous_FDE} in the case $m>1$, and Theorem~\ref{th:acc}. Throughout this section, we are thus equipped with $m>1$, $\alpha >0$, as well as $1 \leq \beta<1+\frac{1}{\alpha}$, and we assume that \eqref{algebraic-acc-kpp}, \eqref{nonlinearity-acc} and \eqref{nonlinearity-accbis} hold.

\subsection{Lower bound on the level sets in \eqref{levelset-kpp} and \eqref{levelset}}\label{ss:lower}

Notice that, in view of \eqref{algebraic-acc-kpp} and the comparison principle, we only need to consider the case where
\begin{equation}
\label{algebraic-acc2}
u_0(x)= \frac{C}{x^\alpha},\quad \forall x\geq x_0.
\end{equation}
By another comparison argument, we can also assume without loss of generality that $u_0\in C^2(\R)$.
\medskip

\noindent{\bf An accelerating small bump as a subsolution.} The main difficulty is to construct a subsolution which has the form of a small bump and travels to the
 right while accelerating. In this $m>1$ regime it actually turns out that the ones constructed in \cite{Ham-Roq-10} for $m=1$, $\beta=1$, and in \cite{Alf-tails} for $m=1$, $\beta>1$ still work. We start with some preparations.

\medskip

Let $\ep>0$ small be given. We first introduce $\eta >0$ such that
\begin{equation*}
\label{enplus}
\beta<1+\eta .
\end{equation*}
Then we select a $\rho>0$ such that
\begin{equation*}
\label{def-rho}
\max\left(\frac{r \beta }{1+\eta},r-\ep\right)<\rho <r.
\end{equation*}
Now define $w(t,x)$ the solution of
\begin{equation*}
\label{edo-w}
\partial _t w(t,x)=\rho w^\beta (t,x),\quad w(0,x)=u_0(x).
\end{equation*}
Then depending on $\beta$, we have either
\begin{equation}
\label{def-w-kpp}
w(t,x):=u_0(x) e^{\rho t},
\end{equation}
if $\beta =1$, or
\begin{equation}\label{def-w}
w(t,x):=\frac{1}{\left(\frac{1}{u_0^{\beta-1}(x)}-\rho (\beta -1)t\right)^{\frac{1}{\beta -1}}}\quad \text{ for } 0\leq t <T(x):=\frac{1}{\rho(\beta-1)u_0^{\beta-1}(x)},
\end{equation}
if $\beta >1$.
\begin{rem} A straightforward computation shows that in both cases, the level sets of $w$ are consistent with the conclusions of Theorems~\ref{th:FKPP_porous_FDE} and \ref{th:acc}. Notice also that, in the case $\beta >1$, the interval of existence $(0,T(x))$ of the solution $w(t,x)$ becomes large as $x\to+\infty$. Indeed, in view of \eqref{algebraic-acc2},
$$
T(x)= \frac{x^{\alpha(\beta -1)}}{\rho(\beta-1)C^{\beta -1}}, \quad \forall x\geq x_0.
$$
In particular, since $0 < \alpha (\beta -1)<1$, we will be able to observe acceleration (a large time phenomenon) in the subdomain where $w$ is well defined. 
\end{rem}
Straightforward computations yield 
\begin{eqnarray*}
\partial _{x} w(t,x)&=&\varphi(x)w^{\beta}(t,x)\\
\partial _{xx} w(t,x)
&=& \varphi'(x)w^{\beta}(t,x)+\beta \varphi^{2}(x)w^{2\beta-1}(t,x),
\end{eqnarray*}
 where 
 \begin{equation}
 \label{varphi}
  \varphi(x):=\frac{u_0'(x)}{u_0^\beta(x)},
 \end{equation}
so that
\begin{eqnarray}
\partial _{x} (w^{m})(t,x)&=&m\varphi(x)w^{m+\beta-1}(t,x) \nonumber \label{calcullaplacien1}\\
\partial _{xx} (w^{m})(t,x)
&=& m\varphi'(x)w^{m+\beta-1}(t,x)+m(m+\beta-1)\varphi^{2}(x)w^{m+2\beta-2}(t,x).\label{calcullaplacien3}
\end{eqnarray}
In view of \eqref{algebraic-acc2} we have
\begin{equation*}\label{defgh-bis}
\varphi'(x)=\frac{\alpha(1-\alpha(\beta-1))}{C^{\beta-1}x^{2-\alpha(\beta-1)}}, \quad \varphi ^{2}(x)=\left(\frac{\alpha}{C^{\beta-1}x^{1-\alpha(\beta-1)}}\right)^{2},  \quad \forall x\geq x_0.
\end{equation*}
Since $\beta<1+\frac 1 \alpha$  both $\varphi'(x)$ and $\varphi ^{2}(x)$ tend to zero as $x\to +\infty$. Let us therefore select~$x_1>x_0$ such that
\begin{equation}
\label{loin1}
m \vert \varphi'(x)\vert  \leq r-\rho \quad \text{ and }\quad m\vert \varphi'(x)\vert+m(2m+\beta+\eta -1)\varphi^2(x) \leq \rho-\frac{r\beta}{1+\eta},\quad \forall x\geq x_1.
\end{equation}
Now, Assumption \ref{ass:initial} implies that
\begin{equation}
\label{def-kappa}
\kappa:=\inf _{x\in(-\infty,x_1)} u_0(x)\in(0,1] .
\end{equation}
Last, we select $A>0$ large enough so that
\begin{equation*}
\label{def-A}
A> \frac{1}{\kappa ^{\eta}},
\end{equation*}
and
\begin{equation}
\label{def-A2}
\frac{\eta}{1+\eta}\frac{1}{(A(1+\eta))^{1/\eta}}\leq s_0,
\end{equation}
where $s_0$ is as in \eqref{nonlinearity-acc}.
Equipped with the above material, we are now in the position to construct the desired subsolution.

\begin{lem}[An accelerating subsolution]
\label{lem:sub} Let $m>1$ and the assumptions of either Theorem~\ref{th:FKPP_porous_FDE} or Theorem~\ref{th:acc} hold. Further assume that \eqref{algebraic-acc2} holds.

Define
$$
v(t,x):=\max\left(0, w(t,x)-Aw^{1+\eta}(t,x)\right).
$$
Then
\begin{equation}
\label{comparison}
v(t,x)\leq u(t,x),\quad \forall (t,x)\in [0,+\infty)\times \R.
\end{equation}
\end{lem}

\begin{proof} Clearly $v(0,x)\leq u_0(x)$. Recalling that $u(t,x)>0$ for all $t>0$, $x\in \R$, it is thus enough to consider the points $(t,x)$ for which $v(t,x)>0$. We therefore need to show
$$
\mathcal L v(t,x):=\partial _t v(t,x)-\partial _{xx}(v^{m})(t,x)-f(v(t,x))\leq 0\quad \text{ when } v(t,x)=w(t,x)-Aw^{1+\eta}(t,x)>0,
$$
and the conclusion follows from the maximum principle on the corresponding subdomain.

Note that $v(t,x) >0$ implies in particular that $w(t,x)<1/A^{1/\eta}<\kappa$
so that $u_0(x)=w(0,x)\leq w(t,x)<\kappa$ since $t\mapsto w(t,x)$ is increasing. In view of the
 definition of $\kappa$ in~\eqref{def-kappa}, this enforces $x\geq x_1$. As a result estimates \eqref{loin1} are
available. On the other hand $v(t,x)\leq \max_{0\leq w\leq
A^{-1/\eta}}w-Aw^{1+\eta}=\frac{\eta}{1+\eta}\frac{1}{(A(1+\eta))^{1/\eta}}\leq
s_0$ by \eqref{def-A2}. Hence, it follows from~\eqref{nonlinearity-acc} that
\begin{eqnarray*}
f(v(t,x))&\geq& rv^{\beta}(t,x) \\
&=& rw^{\beta}(t,x)(1-Aw^{\eta}(t,x))^{\beta}.
\end{eqnarray*}
Then the convexity inequality $(1-Aw^\eta)^\beta\geq 1-A\beta w^\eta$ yields
\begin{equation}\label{ineg1}
f(v(t,x))\geq rw^{\beta}(t,x)-rA\beta w^{\beta+\eta}(t,x).
\end{equation}
Next, we compute
\begin{equation}
\label{ineg2}
\partial _t v(t,x)=\partial _t w(t,x)-A(1+\eta)\partial _tw(t,x)w^{\eta}(t,x)=\rho w^{\beta}(t,x)-A\rho(1+\eta)w^{\beta+\eta}(t,x)
\end{equation}
and, omitting variables,
\begin{eqnarray*}
\partial _{xx}(v^{m})&=&\partial _{xx}(w^{m})(1-Aw^{\eta})^{m}+2\partial _x(w^{m})\partial _x((1-Aw^{\eta})^{m})+w^{m}\partial_{xx}((1-Aw^{\eta})^{m})\nonumber \\
&=& \left(m\varphi'w^{m+\beta-1}+m(m+\beta-1)\varphi^{2}w^{m+2\beta-2}\right)(1-Aw^{\eta})^{m}\nonumber \\
&&+2m\varphi w^{m+\beta -1}\left(-Am\eta \varphi w^{\beta+\eta -1}(1-Aw^{\eta})^{m-1}\right)\nonumber \\
&&+w^{m}\Big ( -Am\eta \varphi 'w^{\beta+\eta-1}(1-Aw^{\eta})^{m-1}\\
&&\quad \quad \quad -Am\eta \varphi(\beta+\eta -1)w^{\beta +\eta -2}\varphi w^{\beta}(1-Aw^{\eta})^{m-1}\\
&&\quad \quad \quad -Am\eta \varphi w^{\beta +\eta -1}(m-1)(-A\eta w^{\eta -1}\varphi w^{\beta})(1-Aw^{\eta})^{m-2}
\Big)\\
&\geq & m\varphi'w^{m+\beta-1}(1-Aw^{\eta})^{m}\\
&&-2Am^{2}\eta \varphi ^{2} w^{m+2 \beta +\eta-2}(1-Aw^{\eta})^{m-1}\\
&& -Am\eta \varphi 'w^{m+\beta+\eta -1}(1-Aw^{\eta})^{m-1}\\
&& -Am\eta(\beta+\eta -1)\varphi ^2 w^{m+2\beta +\eta -2}(1-Aw^{\eta})^{m-1}
\end{eqnarray*}
where we have dropped some nonnegative terms (recall that $m>1$). Next, the crude estimate $(1-Aw^{\eta})^{m-1}\leq 1$ yields
\begin{eqnarray}
-\partial _{xx}(v^{m})
&\leq & m\vert \varphi' \vert w^{m+\beta-1}+2Am^{2}\eta \varphi ^{2} w^{m+2 \beta +\eta-2}+Am\eta \vert \varphi '\vert w^{m+\beta+\eta -1}\nonumber \\
&&+Am\eta(\beta+\eta -1)\varphi ^2 w^{m+2\beta +\eta -2}\nonumber \\
&\leq & m\vert \varphi' \vert w^\beta+2Am^{2}\eta \varphi ^{2} w^{\beta+\eta}+Am\eta \vert \varphi '\vert w^{\beta+\eta } \nonumber \\
&& +Am\eta(\beta+\eta -1)\varphi ^2 w^{\beta +\eta}.\label{ineg3}
\end{eqnarray}
Combining \eqref{ineg2}, \eqref{ineg3} and \eqref{ineg1}, we arrive at 
\begin{align*}
&\mathcal L v(t,x)\leq w^{\beta}(t,x)\left[\rho-r+m\vert \varphi'(x)\vert\right]\\
&\quad \quad \quad\quad + A (1 + \eta) w^{\beta+\eta}(t,x)\left[-\rho+\frac{r\beta}{1+\eta}+m\vert \varphi'(x)\vert+m(2m+\beta+\eta -1)\varphi^2(x)\ \right].
\end{align*}
Thanks to \eqref{loin1}, both bracket terms are nonpositive and we conclude that  $\mathcal{L} v(t,x) \leq 0$. Lemma \ref{lem:sub} is proved.
\end{proof}

The rest of the proof of the lower bounds in \eqref{levelset-kpp} and \eqref{levelset} is now identical to \cite{Alf-tails}.  

\medskip

\noindent {\bf Proof of the lower bound for small $\lambda$.} Equipped with the
 above subsolution, whose role is to \lq\lq lift'' the solution $u(t,x)$ on intervals that enlarge with
 acceleration,  we first prove the lower bound on the level sets $E_\lambda(t)$ when $\lambda$ is small.

Let us fix
$$
0<\theta< \min \{ C, 1/A^{1/\eta} \},
$$
where $C$ is as in \eqref{algebraic-acc2}. We  claim that, for any $t\geq 0$, there is a unique $y_\theta (t)\in \R$ such that
$w(t,y_\theta(t))=\theta$, and moreover $y_\theta(t)$ is given by either
\begin{equation}\label{def-ytheta-kpp}
y_\theta(t):=\left(\frac{C}{\theta}\right)^{1/\alpha}e^{\frac{\rho}{\alpha}t},
\end{equation}
when $\beta =1$, or 
\begin{equation}
\label{def-ytheta}
y_\theta(t):=\left( \left(\frac C\theta \right)^{\beta-1}+\rho C^{\beta -1}(\beta -1)t\right)^{\frac{1}{\alpha(\beta-1)}},
\end{equation}
when $\beta > 1$.

Indeed, since $\theta<1/A^{1/\eta}<\kappa=\inf _{x\in(-\infty,x_1)}u_0(x)$ and since $w(t,x)\geq w(0,x)=u_0(x)$,
for $w(t,y)=\theta$ to hold one needs $y\geq x_1$. But, when $y\geq x_1>x_0$, one can use formula \eqref{algebraic-acc2}
and then solve equation $w(t,y)=\theta$, thanks to expression \eqref{def-w-kpp} or \eqref{def-w}, to find the unique solution~\eqref{def-ytheta-kpp} or~\eqref{def-ytheta}.

Let us now define the open set
$$
\Omega :=\{(t,x) :\,  t>0, x<y_\theta (t)\}.
$$
Let us evaluate $u(t,x)$ on the boundary $\partial \Omega$.
For $t>0$, it follows from \eqref{comparison} that
$$
u(t,y_\theta(t))\geq w(t,y_\theta (t))-Aw^{1+\eta}(t,y_\theta (t))=
\theta -A\theta ^{1+\eta}>0.
$$
On the other hand, for $t=0$ and $x\leq y_\theta (0)=(\frac C \theta)^{1/\alpha}$, we have
$$
u(0,x)\geq \inf _{x\leq (\frac C \theta)^{1/\alpha}}u_0(x)>0,
$$
in view of Assumption \ref{ass:initial}. As a result $\Theta:=\inf _{(t,x)\in \partial\Omega}u(t,x)>0$. Since $\Theta >0$
is a subsolution for equation \eqref{eq}, it follows from the comparison principle that
\begin{equation}
\label{etoile}
u(t,x)\geq \Theta, \quad \forall t\geq 0, \forall x\leq y_\theta(t).
\end{equation}
This implies in particular that, for any $0<\lambda <\Theta$, we have, for all $t\geq t_\lambda$,
\begin{equation*}
\label{small-levelset} \emptyset \neq E_\lambda (t)\subset
(y_\theta(t),+\infty)\subset (x^{-}_\rho (t),+\infty),
\end{equation*}
where either 
\begin{equation*}
 x^{-}_\rho (t):= e^{\frac{\rho}{\alpha} t},
\end{equation*}	
or
\begin{equation*}
 x^{-}_\rho (t):=\left(\rho C^{\beta -1}(\beta -1)t\right)^{\frac{1}{\alpha(\beta-1)}},
 \end{equation*}	
depending on whether $\beta =1 $ or $\beta >1$. Since $\rho>r-\ep$ this implies the lower bounds in \eqref{levelset-kpp} and \eqref{levelset} for $0<\lambda<\Theta$. \qed

 \medskip

\noindent {\bf Proof of the lower bound for any $\lambda\in(0,1)$.}
 Let us now turn to the  case where $\lambda$ is larger than $\Theta$. Let
 $\Theta\leq \lambda <1$ be given. Let us denote by $v(t,x)$ the solution of \eqref{eq} with initial data
 \begin{equation}\label{initial-data-v}
v_0(x):= \begin{cases} \Theta &\text{
if } x\leq -1
\\
-\Theta x &\text{ if } -1<x<0\\
0 &\text{ if  }x\geq 0.
\end{cases}
\end{equation}
By using a subsolution based on Proposition \ref{prop:fin2} as in the proof of \eqref{invasion} in Section \ref{s:prelim} (or by a straightforward extension of \cite[Theorem 4.1]{Med-03} which dealt with Heaviside type initial data), we see that $
\lim _{t\to+\infty}\inf_{x\leq c _0  t} v(t,x)=1$,
for some $c _0  >0$. In particular there is a time $\tau_{\lambda,\ep}>0$ (this time depends on
$\Theta$ and therefore on $\ep$ from the above construction of the small bump subsolution) such that
\begin{equation}
\label{v-grand}
 v(\tau_{\lambda,\ep},x)>\lambda, \quad \forall x\leq 0.
 \end{equation}

 On the other hand, it follows from \eqref{etoile} and the definition \eqref{initial-data-v}
 that
 $$
 u(T,x)\geq v_0(x-y_\theta(T)),\quad \forall T\geq 0, \forall x\in \R,
 $$
 so that the comparison principle yields
 $$
 u(T+\tau,x)\geq v(\tau,x-y_\theta(T)),\quad \forall T\geq 0, \forall \tau \geq 0, \forall x\in \R.
 $$
 In view of \eqref{v-grand}, this implies that
 $$
 u(T+\tau_{\lambda,\ep},x)>\lambda,\quad \forall T\geq 0, \forall x\leq y_\theta (T).
 $$
 Hence, for any $t\geq T^{1}_{\lambda,\ep}:=\max (\tau_{\lambda,\ep},t_\lambda)$, if we pick
  a $x\in E_\lambda(t)$ then the above implies $x>y_\theta(t-\tau_{\lambda,\ep})$. 

In particular, if $\beta =1$ then
$$x > \left( \frac C\theta \right)^{1/\alpha} e^{\frac{\rho}{\alpha} (t - \tau_{\lambda,\ep}) } \geq e^{ \frac{r-\varepsilon}{\alpha} t}$$
for all $t \geq T_{\lambda,\ep}$, with $T_{\lambda,\ep}>0$ sufficiently large (recall that $\rho > r - \ep$). This concludes the proof of the lower bound in \eqref{levelset-kpp}.

On the other hand, if $\beta >1$ then
 $$
 x>\left( \left(\frac C\theta \right)^{\beta-1}-\rho C^{\beta-1}(\beta-1)\tau _{\lambda,\ep}+\rho C^{\beta -1}(\beta -1)t\right)^{\frac{1}{\alpha(\beta-1)}}\geq \left((r-\ep)C^{\beta-1}(\beta-1)t\right)^{\frac{1}{\alpha(\beta-1)}},
 $$
for all $t$ large enough, which also concludes the proof of  the lower bound in \eqref{levelset}. \qed

\subsection{Upper bound on the level sets in \eqref{levelset-kpp} and \eqref{levelset}}\label{ss:upper}

Let $\lambda \in(0,1)$ and  $\ep>0$ small be given.
Up to enlarging $x_0>1$ which appears in \eqref{algebraic-acc-kpp}, we can assume without loss of generality that
\begin{equation}
\label{xzerogrand}
 m \frac{\alpha(1-\alpha(\beta-1))}{\overline C ^{\beta-1}x_0^{2-\alpha(\beta-1)}}+m(m+\beta-1) \left(\frac{\alpha}{\overline C^{\beta-1}x_0^{1-\alpha(\beta-1)}}\right)^{2} \leq \frac \ep 2,
\end{equation}
and $\frac{\overline C}{x_0^{\alpha}}<1$. This is possible since $0<1-\alpha(\beta-1)<2-\alpha(\beta-1)$.

Now in view of \eqref{algebraic-acc-kpp} and the comparison principle, it is enough to prove the upper bound in \eqref{levelset-kpp} and in \eqref{levelset} when
\begin{equation}
\label{algebraic-acc3}
u_0(x)= \frac{\overline C}{x^\alpha},\quad \forall x\geq x_0.
\end{equation}

Let us select
$$
\rho:=\overline r+\frac \ep 2.
$$
We then define
$$
 \psi(t,x):=\min\left(1, w (t,x) \right),
$$
where $w(t,x)$ is defined by either \eqref{def-w-kpp} or \eqref{def-w} depending on whether $\beta=1$ or $\beta >1$. Since $\inf_{x \leq x_0} u_0 (x) > 0$, there exists $T>0$ large enough so that
$$\psi (t,x) = 1 , \quad \forall t \geq T, \forall x \leq x_0 .$$
We claim that $\psi$ is a (generalized) supersolution for equation \eqref{eq} in the domain $(T,+\infty)\times \mathbb{R}$.
As in Section~\ref{s:no-acceleration}, since~$1$ solves \eqref{eq}, it suffices to consider the points $(t,x)$ where $\psi(t,x)=w(t,x)<1$. From our choice of $T$, this implies that $x > x_0$. In view of
$$
\partial _t w(t,x)=\rho w ^\beta (t,x)=(\overline r+\frac \ep 2)w ^\beta(t,x),
$$
together with \eqref{calcullaplacien3} and inequality \eqref{nonlinearity-accbis}, some straightforward computations yield
\begin{align}
 &\partial _t w(t,x)-\partial _{xx}(w^{m})(t,x)-f(w(t,x))\nonumber\\
&\quad \quad \quad\quad\geq \frac \ep 2 w ^\beta(t,x)-m\varphi'(x)w^{m+\beta-1} (t,x)-m(m+\beta-1) \varphi^2(x) w ^{m+2\beta -2}(t,x)\nonumber\\
&\quad \quad \quad\quad\geq w ^\beta(t,x)\left( \frac{\ep}{2}-m\vert \varphi'(x)\vert -m(m+\beta -1) \varphi^2(x) \right)\label{qqch}
\end{align}
since $0<w(t,x)<1$, $m>1$, and where $\varphi(x)=\frac{u_0'(x)}{u_0^{\beta }(x)}$  already appeared in \eqref{varphi}.

In view of expression \eqref{algebraic-acc3}, some straightforward
computations yield that, for any $x\geq x_0$,
\begin{align*}
 &m\vert \varphi'(x)\vert +m(m+\beta-1) \varphi^2(x) \\
&\quad \quad \quad\quad=m\frac{\alpha(1-\alpha(\beta-1))}{\overline C ^{\beta-1}x^{2-\alpha(\beta-1)}}+m(m+\beta-1) \left(\frac{\alpha}{\overline C^{\beta-1}x^{1-\alpha(\beta-1)}}\right)^{2}
\\
&\quad \quad \quad\quad\leq m \frac{\alpha(1-\alpha(\beta-1))}{\overline C ^{\beta-1}x_0^{2-\alpha(\beta-1)}}+m(m+\beta-1) \left(\frac{\alpha}{\overline C^{\beta-1}x_0^{1-\alpha(\beta-1)}}\right)^{2}
\end{align*}
since $0<1-\alpha(\beta-1)<2-\alpha(\beta-1)$. Now by \eqref{xzerogrand}, we get
$$
m\vert \varphi'(x)\vert +m(m+\beta -1)  \varphi^2(x) \leq \frac \ep 2.
$$
It therefore follows from \eqref{qqch}  that, for any $(t,x)\in (T, +\infty) \times \mathbb{R}$ such that $w(t,x)<1$,
$$
\partial _t w(t,x)-\partial _{xx} (w^m) (t,x)-f(w(t,x))
\geq 0,
$$
which proves our claim that $\psi$ is a supersolution of \eqref{eq} in $(T,+\infty) \times \mathbb{R}$.

Finally, since $w$ is increasing in time and $w(0,\cdot) \equiv u_0 (\cdot)$, we have that 
$$u_0 (\cdot) \leq w (T, \cdot),$$
and by assumption $u_0 \leq 1$. Therefore we can apply the comparison principle and conclude that
\begin{equation}\label{comparison2}
u(t,x)\leq\psi(t+T,x)\leq w(t+T,x),\quad \forall (t,x)\in [0,+\infty)\times  \mathbb{R}.
\end{equation}
For $t\geq t_\lambda$, let us pick a $x\in E_\lambda (t)$. It follows from~\eqref{comparison2} that
$w(t+T,x)\geq \lambda$ which, using the expression for $w$ transfers into either
$$
x\leq  \left( \frac{\overline{C}}{\lambda} \right)^{1/\alpha} e^{\frac{\overline r + \frac \ep 2}{\alpha} (t+T)}
< e^{\frac{\overline r + \ep }{\alpha} t} =:x^+(t),
$$
when $\beta = 1$, or
\begin{align*}
x\leq& \left( \left(\frac{\overline C}{\lambda}\right)^{\beta-1}+ \left(\overline r +\frac
\ep 2\right)\overline C^{\beta -1}(\beta -1)(t+T)\right)^{\frac{1}{\alpha(\beta-1)}}\\
<& \left((\overline r +
\ep )\overline C^{\beta -1}(\beta -1)t\right)^{\frac{1}{\alpha(\beta-1)}}=:x^+(t),
\end{align*}
when $\beta >1$, for $t\geq  T_{\lambda,\ep}$ chosen sufficiently large. This concludes the proof of  the upper bound in \eqref{levelset-kpp} (when $m>1$) and in \eqref{levelset}. \qed

\section{Acceleration regime for fast diffusion}\label{s:FDE-kpp}

In this section, we end the proof of Theorem~\ref{th:FKPP_porous_FDE} by now considering the case $0 < m<1$, and also  prove Theorem \ref{th:acc-FDE}. Throughout this section, we will thus take $0<m<1$, $\alpha >0$, $\beta \geq 1$ (further assumptions will come below), and assume that \eqref{algebraic-acc-kpp}, \eqref{nonlinearity-acc-FDE} and \eqref{nonlinearity-accbis-FDE} hold.

\subsection{Heavier tail by fast diffusion}\label{s:heavy}

We start by some short preliminary on the size of the tail of the solution in the fast diffusion regime, which partly explains qualitative differences with the linear diffusion and porous medium diffusion cases.

If $\alpha < \frac 2{1-m}$ then we are satisfied with the  lower estimate of \eqref{algebraic-acc-kpp}, namely
$u_0(x)\geq \frac{C}{x^\alpha}$ for~$x\geq x_0>1$.

On the other hand, if $\alpha \geq \frac{2}{1-m}$ then the fast diffusion equation immediately makes the tail heavier (w.r.t. to both the exponent and the multiplicative constant) at positive time. More precisely, by comparison, we have $u(t,x)\geq v(t,x)$ where $v(t,x)$ solves the fast diffusion equation
$$
\partial _t v=\partial _{xx} (v^m), \quad t>0,\, x\in \R.
$$
A lower bound for the long time behavior of $v$ is provided by  \cite[Theorem 2.4]{Her-Pie-85}: for a given~$T>0$, there is $C(T)>0$ and $x(T) >0$ such that 
$$v(T,x)\geq \frac{C(T)}{x^{\frac{2}{1-m}}}$$ for $x \geq x(T)$. Moreover $C (T)\to+\infty$ as $T\to +\infty$. 

As a result, up to shifting time by $T$ large enough and by a comparison argument, it is enough (as far as the lower bounds of the level sets are concerned) to consider the case of a smooth decreasing  data such that
\begin{equation}
\label{algebraic-gamma}
u_0(x) = \frac{C}{x^{\gamma}}, \quad \forall x\geq x_0, \quad \gamma:=\min\left(\alpha,\frac{2}{1-m}\right),
\end{equation}
for some $x_0>1$ arbitrarily large. Also, as explained above, if $\alpha \geq \frac{2}{1-m}$ then we can  enlarge the constant $C$ without loss of generality.

\subsection{Lower bound on the level sets in \eqref{levelset-kpp} and \eqref{levelset-FDE}}\label{ss:lower-kpp}

In this subsection, we will mostly assume that
$$1 \leq \beta < \min \left( 1+ \frac{1}{\gamma},  m + \frac{2}{\gamma} \right).$$
Comparing this assumption to those of Theorem \ref{th:FKPP_porous_FDE} and Theorem \ref{th:acc-FDE}, this precludes the case~$\beta = 1$ and $\gamma = \frac{2}{1-m}$ which we consider separately at the end of this subsection. We will see though that the difference is only very minor: when $\beta =1 $ and $\gamma = \frac{2}{1-m}$ (or equivalently $\alpha\geq \frac{2}{1-m}$), it is necessary to enlarge the constant $C$ in \eqref{algebraic-acc-kpp}, which is possible in this case as seen in subsection \ref{s:heavy}.\medskip

\noindent{\bf An accelerating subsolution.} Let $\ep >0$ small be given, and
$$\eta  > \beta - 1 \geq 0.$$
For convenience and to make some of our computations simpler, we will also add the condition
\begin{equation}\label{eq:pratique}
\frac{\eta}{1+\eta} > \frac{1}{2}.
\end{equation}
Then let $\rho >0$ be such that
\begin{equation}
\label{def-rho-kppbis}
\max\left(\frac{r \beta }{1+\eta},r-\ep\right) <\rho <r.
\end{equation}
We again define $w(t,x)$ as the solution of
$$\partial_t w (t,x) = \rho w^\beta (t,x), \quad w(0,x)=u_0 (x),$$
which is given by either \eqref{def-w-kpp} if $\beta =1$, or \eqref{def-w} if $\beta >1$.
In this fast diffusion regime ($0<m<1$) a one side compactly supported subsolution $\max(0,w(t,x)-Aw^{1+\eta}(t,x))$ in the spirit of \cite{Ham-Roq-10} or of Section \ref{s:PME} would not work because of terms like $(1-Aw^\eta(t,x))^{m-1}$, $(1-Aw^\eta(t,x))^{m-2}$ which become infinite as $w(t,x)$ approaches $(1/A)^{1/\eta}$. The idea is to cut by a small constant \lq\lq on the left'' rather than by zero. Let us make this precise.

First, we increase $x_0$ without loss of generality so that the three following inequalities hold:
\begin{equation}\label{eq:x0-acc-fde}
 m \frac{\gamma C^{m-\beta} (\gamma + 1 - \gamma \beta)}{x_0^{2+(m-\beta) \gamma}} \leq r - \rho , 
\end{equation}
\begin{equation}\label{eq:x0-acc-fde2}
 2^{1-m}  m \eta \times \frac{\gamma C^{m-\beta}}{x_0^{2 + (m-\beta)\gamma}} \left( 2 m \gamma +  1 + \gamma \eta     +  2 (1-m)   \frac{\eta}{1+\eta}\gamma\right)  \leq  \rho (1 +\eta) -  r \beta  ,
\end{equation}
\begin{equation}\label{eq:x0-acc-fde3}
2^{1-m}  m \eta \left[ \frac{\gamma C^{m-\beta}}{x_0^{2 + (m-\beta)\gamma}} \left(  \gamma +  1 - \gamma \beta   \right)  \\
+ \theta \frac{\gamma^2 C^{2- 2\beta}}{x_0^{2 + 2\gamma (1 -\beta)}}  \right]   \leq \rho (1+\eta) - r \beta.
\end{equation}
where $\theta:= 2m + \beta + \eta -1+(1-m)\frac{2\eta}{1+\eta}$.  This is possible thanks to \eqref{def-rho-kppbis}, using also the fact that $2 + (m-\beta) \gamma >0$ and $1 + \gamma (1- \beta) >0$. We will see below that the rather tedious left-hand terms of these inequalities arise when computing $\partial_{xx} (v^m)$, with $v = w (1-Aw^\eta)$ our accelerating subsolution.
 
Now Assumption \ref{ass:initial} implies that
\begin{equation*}
\label{def-kappa-kpp}
\kappa:=\inf _{x\in(-\infty,x_0)} u_0(x)\in(0,1] .
\end{equation*}
Next, we select $A>1$ large enough so that
\begin{equation*}
\label{def-A-kpp}
A>\frac{1}{\kappa^\eta (1 +\eta) } 
\end{equation*}
and
\begin{equation*}
\label{def-A2-kpp}
\left(\frac{1}{A(1+\eta)} \right)^{1/\eta} \frac{\eta}{1+\eta} \leq s_0,
\end{equation*}
where $s_0$ is as in \eqref{nonlinearity-acc}. Now we define $X(t) \in \mathbb{R}$ such that 
$$w(t,X(t)) = \left( \frac{1}{A(1+\eta)}\right)^{1/\eta}.$$ Since $w^\eta (t,x) \geq u_0^\eta (x) \geq \kappa^\eta > \frac{1}{A(1+\eta)}$ for all $x \leq x_0$, we have that $X(t) > x_0$. Using the explicit expressions for $w(t,\cdot)$ and $u_0$ on $(x_0,+\infty)$ provided by \eqref{algebraic-gamma}, it is then straightforward that such an $X(t)$ is uniquely defined. Moreover, $w(t,x) <( \frac{1}{A(1+\eta)})^{1/\eta}$ if and only if $x > X(t)$.

\begin{lem}[An accelerating subsolution]
\label{lem:sub-kpp} Let the assumptions of either Theorem~\ref{th:FKPP_porous_FDE} (with $0<m<1$) or Theorem \ref{th:acc-FDE} hold. Further assume that \eqref{algebraic-gamma} holds.

Define
\begin{equation*}
\label{def-sub-kpp}
v(t,x):=
\begin{cases}
\left(\frac{1}{A(1+\eta)} \right)^{1/\eta} \frac{\eta}{1+\eta} &\mbox{ if } x\leq X(t)\\
w(t,x)(1-Aw^\eta (t,x)) &\mbox{ if } x>X(t).
\end{cases}
\end{equation*}
Then $v(t,x)\leq u(t,x)$ for all $t>0$, $x\in \R$.
\end{lem}

\begin{proof} Clearly $v(0,x) < u_0(x)$. Let us already note that $v$ is smooth in both subdomains $\{ x < X(t)\}$ and $\{ x > X(t) \}$. Also, it is continuous in $[0,+\infty) \times \mathbb{R}$ as well as $C^1$ with respect to $x$ at the junction point $X(t)$. This means that a comparison principle is applicable provided that $v$ satisfies 
\begin{equation}\label{plus1}
\mathcal L v(t,x):=\partial _t v(t,x)-\partial _{xx}(v^m)(t,x)-f(v(t,x))\leq 0
\end{equation}
in both these subdomains. Indeed, if $v$ is a positive subsolution on the half-domain $\{ (t,x) \in (0,+\infty) \times (-\infty, X(t))\}$, then either $v(t,x) < u(t,x)$ for all $t >0$ and $x \leq X(t)$, or there exists a first time $T>0$ such that $u(T,X(T))=v(T,X(T)) >0 $. In the latter case, by Hopf lemma we have $\partial_x u (T,X(T)) < \partial_x v(T,X(T))$, which due to the $C^1$-regularity of $u$ and $v$ in the $x$-variable contradicts the comparison principle on the right half-domain $\{ (t,x) \in (0,T) \times (X(t) , +\infty)\}$. Therefore, the inequality $v(t,x) \leq u(t,x)$ holds for all $t \geq 0$ and $x \leq X(t)$, and by the comparison principle it also holds for all $t\geq 0$ and $x > X(t)$.

Let us now check that $v$ satisfies~\eqref{plus1} on both subdomains. Since $\left(\frac{1}{A(1+\eta)} \right)^{1/\eta} \frac{\eta}{1+\eta}$ is obviously a subsolution to \eqref{eq}, we only need to check this inequality when $x > X(t)$. Recall that $X(t) > x_0$, hence \eqref{algebraic-gamma} is available. On the other hand it is straightforward that $v(t,x)\leq \max_{w \geq 0} w (1- Aw^\eta) = \left(\frac{1}{A(1+\eta)} \right)^{1/\eta} \frac{\eta}{1+\eta}\leq s_0$. It then follows from \eqref{nonlinearity-acc} and a  convexity inequality that
\begin{equation}
\label{un}
f(v(t,x))\geq rw^\beta (t,x)-rA \beta w^{\beta + \eta }(t,x).
\end{equation}
Next we have
\begin{equation}
\label{deux}
\partial _t v(t,x)=\rho w^\beta (t,x)-A\rho (1+\eta) w^{\beta + \eta }(t,x).
\end{equation}
Also, by the same computations as in Section \ref{s:PME},
\begin{eqnarray*}
\partial _{xx}(v^{m})&=&\partial _{xx}(w^{m})(1-Aw^{\eta})^{m}+2\partial _x(w^{m})\partial _x((1-Aw^{\eta})^{m})+w^{m}\partial_{xx}((1-Aw^{\eta})^{m})\nonumber \\
&\geq & m\varphi'w^{m+\beta-1}(1-Aw^{\eta})^{m} -2Am^{2}\eta \varphi ^{2} w^{m+2 \beta +\eta-2}(1-Aw^{\eta})^{m-1}\\
&& -Am\eta \varphi 'w^{m+\beta+\eta -1}(1-Aw^{\eta})^{m-1}\\
&& -Am\eta(\beta+\eta -1)\varphi ^2 w^{m+2\beta +\eta -2}(1-Aw^{\eta})^{m-1}\\
& & - (1-m) A^2 m \eta^2 \varphi^2 w^{m+2\beta + 2\eta -2}  (1-Aw^\eta)^{m-2},
\end{eqnarray*}
where $\varphi = \frac{u_0 '}{u_0^\beta}$. Comparing with our argument in Section~\ref{s:PME}, we have here an additional term due to the fact that $m<1$ now. Using $0<m<1$ and $1 \geq 1 - A w^\eta (t,x) \geq \frac{\eta}{1+\eta} \geq \frac{1}{2}$ (recall~\eqref{eq:pratique}), we then obtain that
\begin{eqnarray}
- \partial _{xx}(v^{m})& \leq &  m |\varphi' |w^{m+\beta-1} +2^{2-m} Am^{2}\eta \varphi ^{2} w^{m+2 \beta +\eta-2} \nonumber \\
&& + 2^{1-m} Am\eta |\varphi ' | w^{m+\beta+\eta -1}  +  2^{1-m} Am\eta(\beta+\eta -1)\varphi ^2 w^{m+2\beta +\eta -2} \nonumber \\
& & +  2^{2-m}(1-m)  A^2 m \eta^2 \varphi^2  w^{m+2\beta + 2\eta -2}\nonumber\\
& = &  w^{\beta}m |\varphi' |  (\frac{w}{u_0})^{m-1} u_0^{m -1} + w^{\beta + \eta} \times 2^{2-m} Am^{2}\eta \varphi ^{2}  (\frac{w}{u_0})^{m+ \beta -2} u_0^{m+\beta - 2}\nonumber\\
&& +  w^{\beta + \eta} \times 2^{1-m} Am\eta |\varphi ' |(\frac{w}{u_0})^{m -1}u_0^{m-1}\nonumber\\
&&  +  w^{\beta +\eta} \times 2^{1-m} Am\eta(\beta+\eta -1)\varphi ^2  (\frac{w}{u_0})^{m+\beta  -2} u_0^{m+ \beta - 2}\nonumber\\
& & +w^{\beta+\eta} \times 2^{2-m}(1-m)   A m \eta^2 (Aw^{\eta})\varphi^2 (\frac{w}{u_0})^{m+\beta-2}u_0^{m+\beta-2}\label{eq:acc-sub-test}
\end{eqnarray}
as well as
\begin{eqnarray}
- \partial_{xx} (v^m) & \leq &  w^{\beta}m |\varphi' |  (\frac{w}{u_0})^{m-1} u_0^{m -1} + w^{\beta + \eta} \times 2^{2-m} Am^{2}\eta \varphi ^{2}  w^{m+ \beta -2} \nonumber\\
&& +  w^{\beta + \eta} \times 2^{1-m} Am\eta |\varphi ' |(\frac{w}{u_0})^{m -1}u_0^{m-1} \nonumber\\
&& +  w^{\beta +\eta} \times 2^{1-m} Am\eta(\beta+\eta -1)\varphi ^2  w^{m+\beta  -2} \nonumber\\
& & +w^{\beta+\eta} \times 2^{2-m}(1-m)   A m \eta^2 (Aw^{\eta})\varphi^2 w^{m+\beta-2}.\label{eq:acc-sub-test-bis}
\end{eqnarray}
We distinguish below the two cases $\beta \leq 2 -m$ and $\beta > 2-m$. This is rather natural since~$2-m$ is the value which both hyperbolae $1 + \frac{1}{\gamma}$ and $m + \frac{2}{\gamma}$ take at their intersection point $\gamma = \frac{1}{1-m}$.\smallskip

$\bullet$ Let us first consider the case $\beta \leq 2 -m$ for which we will take advantage of \eqref{eq:acc-sub-test}. Since~$m<1$ and $w (t,\cdot) \geq u_0 (\cdot)$ for any $t \geq 0$, we have that
\begin{equation}\label{eq:wu0-sub}
(\frac{w}{u_0})^{m-1} , (\frac{w}{u_0})^{m+\beta - 2} \leq 1.
\end{equation}
Recall that $x \geq x_0$ and $u_0 (x) = \frac{C}{x^\gamma}$. Thus a straightforward computation leads to
\begin{equation}\label{eq:phiFDE1-sub}
\varphi^2 (x) u_0^{m+\beta -2} (x) = \frac{\gamma^2 C^{m - \beta}}{x^{2 + (m-\beta) \gamma}} \leq \frac{\gamma^2 C^{m - \beta}}{x_0^{2 + (m-\beta) \gamma}}.
\end{equation}
Here we used the fact that $\beta < m + \frac{2}{\gamma}$, so that $2 + (m-\beta) \gamma> 0$. Similarly, we also get
\begin{equation}\label{eq:phiFDE2-sub}
\varphi' (x) u_0^{m-1}(x) = \frac{\gamma C^{m-\beta} (\gamma + 1 - \gamma \beta)}{x^{2+ (m-\beta) \gamma}} \leq  \frac{\gamma C^{m-\beta} (\gamma + 1 - \gamma \beta)}{x_0 ^{2+ (m-\beta) \gamma}}.\end{equation} 
Finally, we also have
\begin{equation}\label{eq:encore}
(A w^{\eta}(t,x))\varphi^2 (x) u_0^{m+\beta  - 2} (x) \leq \frac{1}{1+\eta} \varphi^2 (x) u_0^{m+\beta  - 2} (x)\leq \frac{1}{1+\eta}\frac{\gamma^2 C^{m - \beta}}{x_0^{2 + (m-\beta) \gamma}},
\end{equation}
since  $w (t,x) \leq ( \frac{1}{A(1+\eta)})^{1/\eta} \leq 1$ for $t >0$ and $x \geq X(t)$ and from \eqref{eq:phiFDE1-sub}. Plugging  inequalities \eqref{eq:wu0-sub}, \eqref{eq:phiFDE1-sub}, \eqref{eq:phiFDE2-sub} and \eqref{eq:encore} into \eqref{eq:acc-sub-test}, we get
\begin{eqnarray*}
- \partial _{xx}(v^{m})
& \leq &  w^\beta \left[ m \frac{\gamma C^{m-\beta} (\gamma + 1 - \gamma \beta)}{x_0^{2+(m-\beta) \gamma}} \right] \\
&&  + w^{\beta + \eta} \left[2^{2-m} Am^{2}\eta \frac{\gamma^2 C^{m-\beta}}{x_0^{2+(m - \beta)\gamma}} + 2^{1-m} Am\eta \frac{\gamma C^{m-\beta} (\gamma + 1 - \gamma \beta)}{x_0^{2+(m-\beta) \gamma}} \right.\\
&& \left. +  2^{1-m} Am\eta(\beta+\eta -1)      \frac{\gamma^2 C^{m-\beta}}{x_0^{2+(m - \beta)\gamma}} +   2^{2-m}(1-m) A m \eta^2 \frac{1}{1+\eta}\frac{\gamma^2 C^{m - \beta}}{x_0^{2 + (m-\beta) \gamma}}  \right] . \\
& \leq &  w^\beta \left[ m \frac{\gamma C^{m-\beta} (\gamma + 1 - \gamma \beta)}{x_0^{2+(m-\beta) \gamma}} \right] \\
&&  + w^{\beta + \eta}  \times  2^{1-m}A   m \eta \left[ \frac{\gamma C^{m-\beta}}{x_0^{2 + (m-\beta)\gamma}} \left( 2 m \gamma +  1 + \gamma \eta     +  2 (1-m)   \frac{\eta}{1+\eta}\gamma\right)  \right] . \\
\end{eqnarray*}
By \eqref{eq:x0-acc-fde} and \eqref{eq:x0-acc-fde2}, we conclude that
\begin{equation}\label{truc}
-\partial_{xx} (v^m) \leq (r-\rho) w^\beta  +  A(\rho (1+\eta) -  r \beta) w^{\beta + \eta}.
\end{equation}
Together with \eqref{un}, \eqref{deux}, it now  follows that
\begin{equation}\label{truc2}
\partial _t v(t,x)-\partial _{xx}(v^{m})(t,x)-f(v(t,x))\leq 0,
\end{equation}
for all $t \geq 0$ and $x > X(t)$. This proves Lemma~\ref{lem:sub-kpp} when $\beta \leq 2-m$. 

$\bullet$  Let us now turn to the case $\beta > 2-m$ for which we will take advantage of \eqref{eq:acc-sub-test-bis}. Since~$m<1$ and $1> w (t,\cdot) \geq u_0 (\cdot)$ for any $t \geq 0$, we have that
\begin{equation}\label{eq:wu0-sub-bis}
(\frac{w}{u_0})^{m-1} , w^{m+\beta - 2} \leq 1.
\end{equation}
Notice that \eqref{eq:phiFDE2-sub} still holds. In this regime rather than \eqref{eq:phiFDE1-sub} and \eqref{eq:encore} we use 
\begin{equation}\label{eq:phiFDE1-subbis}
\varphi^2 (x)  \leq \frac{\gamma^2 C^{2-2\beta}}{x_0^{2+2\gamma (1-\beta)}}.
\end{equation}
Plugging \eqref{eq:wu0-sub-bis}, \eqref{eq:phiFDE2-sub} and \eqref{eq:phiFDE1-subbis} into \eqref{eq:acc-sub-test-bis}, we get
\begin{eqnarray*}
- \partial_{xx} (v^{m}) & \leq & w^\beta \left[ m \frac{\gamma C^{m-\beta} (\gamma + 1 - \gamma \beta)}{x_0^{2+(m-\beta) \gamma}} \right] \\
&&  + w^{\beta + \eta} \left[2^{2-m} Am^{2}\eta \frac{\gamma^2 C^{2-2\beta}}{x_0^{2+ 2 \gamma (1 - \beta)}} + 2^{1-m} Am\eta \frac{\gamma C^{m-\beta} (\gamma + 1 - \gamma \beta)}{x_0^{2+(m-\beta) \gamma}} \right.\\
&& \left. +  2^{1-m} Am\eta(\beta+\eta -1)      \frac{\gamma^2 C^{2-2\beta}}{x_0^{2+2 \gamma (1 - \beta)}} + 2^{2-m}(1-m)  A m \eta^2 \frac{1}{1+\eta} \frac{\gamma^2 C^{2-2\beta}}{x_0^{2 + 2\gamma (1 -\beta)}}  \right] \\
& = &  w^\beta \left[ m \frac{\gamma C^{m-\beta} (\gamma + 1 - \gamma \beta)}{x_0^{2+(m-\beta) \gamma}} \right] \\
&&  + w^{\beta + \eta}  \times  2^{1-m}  A m \eta \left[ \frac{\gamma C^{m-\beta}}{x_0^{2 + (m-\beta)\gamma}} \left(  \gamma +  1 - \gamma \beta   \right)  +  \theta \frac{\gamma^2 C^{2- 2\beta}}{x_0^{2 + 2\gamma (1 -\beta)}}  \right], 
\end{eqnarray*}
where $\theta= 2m + \beta + \eta -1+(1-m)\frac{2\eta}{1+\eta}$. By \eqref{eq:x0-acc-fde} and \eqref{eq:x0-acc-fde3}, we see that \eqref{truc} still holds
and again obtain \eqref{truc2}. This
 concludes the proof of Lemma~\ref{lem:sub-kpp}. \end{proof}

\medskip

\noindent{\bf The case $\beta=1$ and $\gamma = \frac{2}{1-m}$.} The only difference is that it is not sufficient to enlarge $x_0$ in order to obtain \eqref{eq:x0-acc-fde} and \eqref{eq:x0-acc-fde2}, since the $x_0$-exponent $2 +(m-\beta) \gamma$ is now null. However, as it has been explained in subsection~\ref{s:heavy}, in this case and up to a finite shift in time, we can assume without loss of generality that the constant $C$ is arbitrarily large. Then, $C^{m-\beta}$ can be made arbitrarily small so that \eqref{eq:x0-acc-fde} and \eqref{eq:x0-acc-fde2} again hold. The rest of the proof remains unchanged.

\medskip

\noindent {\bf Proof of the lower bound on the level sets.} Now that we have an explicit accelerating subsolution, the proof is the same as that in subsection~\ref{ss:lower} for the $m>1$ case. We omit the details and conclude that the lower bounds in \eqref{levelset-kpp} (when $m<1$) and \eqref{levelset-FDE} are proved.

\subsection{Upper  bound on the level sets  in \eqref{levelset-kpp} and \eqref{levelset-FDE}}\label{ss:upper-kpp}

Let $\lambda \in(0,1)$ and $\ep >0$ small be given. Recall that
$\gamma=\min\left(\alpha,\frac{2}{1-m}\right)$ and assume that
\begin{equation}\label{eq:beta-plus}
1\leq \beta < \min \left( 1 + \frac{1}{\gamma}, m + \frac{2}{\gamma} \right).
\end{equation}
As in subsection \ref{ss:lower-kpp}, this precludes the case $\beta=1$ and $\gamma = \frac{2}{1-m}$, which we again consider separately at the end of this subsection.  

{}From \eqref{algebraic-acc-kpp}, $\gamma \leq \alpha$ and the comparison principle, it is enough (to prove the upper estimate on the level sets) to consider the case where
\begin{equation*}
\label{algebraic-encore}
u_0(x)=\frac{\overline C}{x^{\gamma}}, \quad \forall x\geq x_0>1 .
\end{equation*}

Furthermore, without loss of generality we assume that $x_0$ is large enough so that $\frac{\overline C}{x_0^{\gamma}} < 1$ and
\begin{equation}\label{C-bar-2}
\max \left( m \frac{\gamma \overline{C}^{m-\beta} (\gamma + 1 - \gamma \beta)}{x_0 ^{2+ (m-\beta) \gamma}}
, \,  m (m+\beta -1) \frac{\gamma^2 \overline{C}^{m - \beta}}{x_0^{2 + (m-\beta) \gamma}} , \,  m (m+\beta -1) \frac{\gamma^2 \overline{C}^{2 - 2\beta}}{x_0^{2 \gamma (1-\beta) +2}} \right) \leq \frac{\ep}{4}.
\end{equation}
Such an $x_0$ exists thanks to \eqref{eq:beta-plus}.

Let us select $\rho:=\overline r +\frac \ep 2$. Similarly as in Section \ref{s:PME}, we define $w(t,x)$ by either \eqref{def-w-kpp} or \eqref{def-w}, so that it solves 
$$\partial_t w (t,x) = \rho w^\beta (t,x), \quad w (0,x) = u_0 (x).$$
We prove below that
$$
\psi(t,x):=\min\left(1,w(t,x)\right)
$$
is a supersolution for equation \eqref{eq} after some large enough time $T>0$. More precisely, let~$T>0$ be large enough so that 
$$\psi (t,x) = 1, \quad \forall t \geq T, \forall x \leq x_0,$$
which is possible thanks to the fact that the infimum of $u_0$ on $(-\infty,x_0]$ is positive. Since~1 solves~\eqref{eq}, we only need to check that $\psi (t,x)$ is a supersolution at the points $(t,x)$ where $\psi(t,x)=w(t,x)<1$. In view of \eqref{nonlinearity-accbis-FDE}, we get, as in \eqref{qqch}, 
\begin{align}
 &\partial _t w(t,x)-\partial _{xx}(w^{m})(t,x)-f(w(t,x))\nonumber\\
&\quad \quad \quad\quad\geq \frac \ep 2 w ^\beta(t,x)-m\varphi'(x)w^{m+\beta-1} (t,x)-m(m+\beta-1) \varphi^2(x) w ^{m+2\beta -2}(t,x)\nonumber\\
&\quad \quad \quad\quad\geq w ^\beta(t,x)\Big( \frac{\ep}{2}-m \varphi' (x) \left( \frac{w(t,x)}{u_0(x)}\right)^{m-1} u_0^{m-1}(x)\nonumber\\
& \quad \quad \quad\quad \quad \quad \quad \quad \quad - m (m+\beta -1) \varphi^2 (x) \left(\frac{w(t,x)}{u_0(x)}\right)^{m+\beta -2} u_0^{m+\beta-2}(x) \Big),\label{qqchnew}
\end{align}
where we recall  that $\varphi(x)= \frac{u_0 '(x)}{u_0^\beta(x)}$. We again distinguish below the two cases $\beta \leq 2 -m$ and~$\beta > 2-m$. \smallskip

$\bullet$ Let us first consider the case when $\beta \leq 2 -m$. Since $m< 1$ and $w (t,\cdot) \geq u_0 (\cdot)$ for any~$t \geq 0$, we have that
\begin{equation}\label{eq:wu0}
\left(\frac{w(t,x)}{u_0(x)}\right)^{m-1} , \left(\frac{w(t,x)}{u_0(x)}\right)^{m+\beta - 2} \leq 1.
\end{equation}
{}From our choice of $T$, we only need to consider $x \geq x_0$, so that $u_0 (x) = \frac{\overline{C}}{x^\gamma}$. As before, it is straightforward to compute that
\begin{equation}\label{eq:phiFDE1}
\varphi^2 (x) u_0^{m+\beta -2} (x) = \frac{\gamma^2 \overline{C}^{m - \beta}}{x^{2 + (m-\beta) \gamma}} \leq \frac{\gamma^2 \overline{C}^{m - \beta}}{x_0^{2 + (m-\beta) \gamma}},
\end{equation}
as well as
\begin{equation}\label{eq:phiFDE2}
\varphi' (x) u_0^{m-1}(x) = \frac{\gamma \overline{C}^{m-\beta} (\gamma + 1 - \gamma \beta)}{x^{2+ (m-\beta) \gamma}} \leq  \frac{\gamma \overline{C}^{m-\beta} (\gamma + 1 - \gamma \beta)}{x_0 ^{2+ (m-\beta) \gamma}},\end{equation}
thanks to the fact that $\beta < m + \frac{2}{\gamma}$. Putting \eqref{qqchnew} together with \eqref{eq:wu0}, \eqref{eq:phiFDE1} and \eqref{eq:phiFDE2}, we obtain that
\begin{align}
 &\partial _t w(t,x)-\partial _{xx}(w^{m})(t,x)-f(w(t,x))\nonumber\\
&\quad \quad \quad\quad\geq w ^\beta(t,x)\left( \frac{\ep}{2}-m \frac{\gamma \overline{C}^{m-\beta} (\gamma + 1 - \gamma \beta)}{x_0 ^{2+ (m-\beta) \gamma}}
- m (m+\beta -1) \frac{\gamma^2 \overline{C}^{m - \beta}}{x_0^{2 + (m-\beta) \gamma}} \right),\nonumber
\end{align}
for all $t \geq T$ and $x \in \mathbb{R}$ such that $w(t,x) < 1$. This is nonnegative by \eqref{C-bar-2}, which proves that~$\psi$ is a supersolution in $(T,+\infty)\times \R$ when $\beta \leq 2-m$. 

$\bullet$ Let us now consider the case when $\beta > 2-m$. Then we still have that $(\frac{w(t,x)}{u_0(x)})^{m-1} \leq 1$, and \eqref{eq:phiFDE2} holds. On the other hand, $m + 2 \beta - 2 > \beta$ and thus $w^{m+ 2 \beta -2} < w^\beta $. Going back to  the second line of~\eqref{qqchnew} and replacing also $\varphi^2 (x)$ by its explicit expression when $x \geq x_0$, one arrives at
\begin{align}
 &\partial _t w(t,x)-\partial _{xx}(w^{m})(t,x)-f(w(t,x))\nonumber\\
&\quad \quad \quad\quad\geq w ^\beta(t,x)\left( \frac{\ep}{2}-m \frac{\gamma \overline{C}^{m-\beta} (\gamma + 1 - \gamma \beta)}{x_0 ^{2+ (m-\beta) \gamma}}
- m (m+\beta -1) \frac{\gamma^2 \overline{C}^{2 - 2\beta}}{x_0^{2 \gamma (1-\beta) +2}} \right). \nonumber
\end{align}
Notice that we have bounded $\varphi^2 (x)$ by its value at $x=x_0$, which is possible because we assumed $2 \gamma (1-\beta) + 2 >0$. Thanks to \eqref{C-bar-2}, we also conclude that $\psi$ is a supersolution in $(T,+\infty) \times \R$ when $\beta > 2-m$. \medskip

Let us now proceed with the proof on the upper bound on the level sets, which is now identical to the case $m>1$. Since $u_0 \equiv  w(0,\cdot)$ and $w$ is increasing in time, it is clear that $u_0 (\cdot) \leq \psi (T,\cdot)$. Applying a comparison principle, it follows that
\begin{equation*}\label{comparisonbis}
u(t,x)\leq\psi(t+T,x)\leq w(t+T,x),\quad \forall (t,x)\in [0,+\infty)\times  \mathbb{R}.
\end{equation*}
For $t\geq t_\lambda$ and $x\in E_\lambda (t)$, it follows that
$w(t+T,x)\geq \lambda$ which, using the expression for $w$ transfers into either
\begin{align*}
x\leq  \left( \frac{\overline{C}}{\lambda} \right)^{1/\gamma} e^{\frac{\overline r + \frac \ep 2}{\gamma} (t+T)} <  e^{\frac{\overline r + \ep }{\gamma} t} =:x^+(t),
\end{align*}
when $\beta = 1$, or
\begin{align*}
x\leq& \left( \left(\frac{\overline C}{\lambda}\right)^{\beta-1}+ \left(\overline r +\frac
\ep 2\right)\overline C^{\beta -1}(\beta -1)(t+T)\right)^{\frac{1}{\gamma(\beta-1)}}\\
<& \left((\overline r +
\ep )\overline C^{\beta -1}(\beta -1)t\right)^{\frac{1}{\gamma(\beta-1)}}=:x^+(t),
\end{align*}
when $\beta >1$, for $t\geq  T_{\lambda,\ep}$ chosen sufficiently large. This concludes the proof of  the upper bound in \eqref{levelset-kpp} (when $0<m<1$) since $\frac 1 \gamma= \Gamma$, and  in~\eqref{levelset-FDE} since $\gamma=\alpha$ when $0<\alpha<\frac{2}{1-m}$.  \medskip

\noindent{\bf The case $\beta=1$ and $\gamma= \frac{2}{1-m}$.} Again, the argument only slightly differs due to the fact that the exponents $2+(m-\beta)\gamma$ of $x_0$ appearing in \eqref{C-bar-2} now vanish. Hence, in order for \eqref{C-bar-2} to hold, we need to enlarge not only $x_0$ but also $\overline{C}$ (notice in \eqref{C-bar-2} the term $\overline C^{m-\beta}=\overline C^{m-1}$ which is small when $\overline C$ is large). This has no incidence because in the Fisher-KPP case, the constant $\overline{C}$ does not appear in the upper bound \eqref{levelset-kpp} on the level sets. The remainder of the argument is as above and we omit the details.\qed

\appendix

\section{Lower acceleration bound \eqref{dernier-truc} for fast diffusion}\label{s:appendix}

Here we prove statement $(ii)$ of Theorem~\ref{th:reste}, which provides a (a priori not optimal) polynomial lower bound on the level sets of solutions in the fast diffusion case $0<m<1$. Throughout this section we assume that \eqref{algebraic-acc-FDE} holds with
$$\frac{1}{1-m} < \alpha \leq \frac{2}{1-m},$$
and also that
$$\beta >1 , \quad  m+ \frac{2}{\alpha} \leq   \beta <  1+ \frac{1}{\alpha}.$$
Let us note that the case when \eqref{algebraic-acc-kpp-cor} holds easily follows from the space asymptotics of solutions of the fast diffusion equation, as explained in subsection~\ref{s:heavy}. Thus we omit this situation here.

As before, the proof relies on the construction of an accelerating subsolution. This subsolution has the same shape than that of Section~\ref{s:FDE-kpp}, however technical details will differ. Let~$\ep >0$ arbitrarily small be given, and
$$\eta  > \beta +2.$$
Then let $\rho >0$ be such that
\begin{equation}
\label{def-rho-kppter}
\max\left(\frac{r \beta }{1+\eta},r-\ep\right) <\rho <r.
\end{equation}
By comparison and without loss of generality, we assume that $u_0 (x) = \frac{C}{x^\alpha}$ for all $x \geq x_0$. Let us also introduce an auxiliary initial datum
$$\tilde{u}_0(x): = \frac{C - \varepsilon}{C} u_0(x),$$
so that for all $x \geq x_0$,
\begin{equation}\label{algebraic-gammater}
\tilde{u}_0  (x)= \frac{C- \varepsilon}{x^\alpha}.
\end{equation}
Similarly as before, we define $w(t,x)$ as the solution of
$$\partial_t w (t,x) = \rho w^\beta (t,x), \quad w(0,x)= \tilde{u}_0 (x),$$
which is given by \eqref{def-w}, up to replacing $u_0(x)$ by $\tilde{u}_0(x)$.

Now select $A>1$ large enough so that
\begin{equation*}
A> \frac{1}{\kappa^\eta (1 +\eta) } , \quad
\left(\frac{1}{A(1+\eta)} \right)^{1/\eta} \frac{\eta}{1+\eta} \leq s_0,
\end{equation*}
where
\begin{equation}
\label{def-kappa-kppter}
\kappa:=\inf _{x\in(-\infty,x_0)} \tilde{u}_0(x)\in(0,1],
\end{equation}
and $s_0$ is as in \eqref{nonlinearity-acc}. Then, as in subsection~\ref{ss:lower-kpp}, we define $X(t) \in \mathbb{R}$ such that
$$w(t,X(t)) = \left( \frac{1}{A(1+\eta)}\right)^{1/\eta}.$$ 
{}From our choice of $A$, we still have that $X(t)> x_0$ and, from the explicit expressions for~$w(t,\cdot)$ and $\tilde{u}_0$ on $(x_0,+\infty)$, it is straightforward that $X(t)$ is uniquely defined and $w(t,x) < \left( \frac{1}{A(1+\eta)}\right)^{1/\eta}$ if and only if $x > X(t)$. Let us also note that $X(t)$ is increasing with respect to time and tends to $+\infty$ as $t \to +\infty$.

Now let also $\delta$ be small enough so that
\begin{equation}\label{def-delta1-ter}
 \delta <  \left(\frac{1}{A(1+\eta)}\right)^{1/\eta},
\end{equation}
and
\begin{equation}\label{def-delta2-ter}
\frac{A \delta^\eta}{1 - A \delta^\eta} \times \left( 2  m \eta  +  \eta (\beta + \eta -1)   + (1-m) A \eta^2 \frac{\delta^{\eta}}{1- A \delta^\eta} \right)<  m + \beta -1 . 
\end{equation}
We are now ready to construct a subsolution.

\begin{lem}
\label{lem:sub-kppter} Under the above assumptions, define
\begin{equation*}
v(t,x):=
\begin{cases}
\left(\frac{1}{A(1+\eta)} \right)^{1/\eta} \frac{\eta}{1+\eta} &\mbox{ if } x\leq X(t)\\
w(t,x)(1-Aw^\eta (t,x)) &\mbox{ if } x>X(t).
\end{cases}
\end{equation*}
Then there exists $T\geq 0 $ such that $v(t,x)$ is a (generalized) subsolution for all $t \geq T$ and~$x \in \R$.
\end{lem}
By generalized subsolution, we mean that $v$ is $C^1$ with respect to $x$ and satisfies
$$\mathcal{L} v (t,x) := \partial_t v (t,x) - \partial_{xx} (v^m) (t,x) - f(v(t,x)) \leq 0$$
on both subdomains $\{ x < X(t)\}$ and $\{x > X(t)\}$. This is enough to apply a comparison principle, as we explained in subsection~\ref{ss:lower-kpp}. 

\begin{proof} The $C^1$-regularity of $v$ is straightforward, as $w \mapsto w(1-Aw^\eta)$ reaches its maximum when $w= (\frac{1}{A(1+\eta)})^{1/\eta}$. Moreover, any positive constant is a subsolution, so that we only need to check the parabolic inequality $\mathcal{L} v \leq 0$ on the right subdomain, i.e. when $x > X(t)$ (and~$t \geq T$ to be chosen below).

First, recall that $X(t) > x_0$ for any $t \geq 0$. In particular, \eqref{algebraic-gammater} is available. Similarly as before, we define $\varphi = \frac{\tilde{u}_0 '}{\tilde{u}_0^\beta}$, and compute 
$$\varphi^2 (x) = \frac{\alpha^2}{(C-\varepsilon)^{2\beta-2} x^{2 (\alpha + 1 - \alpha \beta)}} , \qquad \varphi ' (x) = \frac{\alpha (\alpha + 1 - \alpha \beta)}{(C-\varepsilon)^{\beta-1} x^{\alpha + 2 - \alpha \beta}}.$$
We know by assumption that $\alpha \beta < \alpha +1$. Since $X (t) \to +\infty$ as $t \to +\infty$, we can choose $T$ large enough so that, for all $t \geq T$ and $x \geq X(t)$, we have $\varphi^2$ and $\varphi'$ small enough so that
\begin{equation}\label{phi-estimter}
0<   \varphi ^{2}  A^{-\frac{m+  \beta -2}{\eta}} \Big(2 m +   \beta+\eta -1    + 2 (1-m)    \eta \Big)  +     \varphi ' A^{-\frac{m  -1}{\eta}}   \leq \frac{r - \rho}{2^{1-m} m \eta}.
\end{equation}
We are now ready to compute $\mathcal{L} v$. First, from our choice of $A$, \eqref{nonlinearity-acc-FDE} and a convexity inequality, we have that
\begin{equation*}
f(v(t,x))\geq rw^\beta (t,x)-rA \beta w^{\beta + \eta }(t,x).
\end{equation*}
Next we have
\begin{equation*}
\partial _t v(t,x)=\rho w^\beta (t,x)-A\rho (1+\eta) w^{\beta + \eta }(t,x),
\end{equation*}
and by the same computations as in Section \ref{s:PME},
\begin{eqnarray*}
\partial _{xx}(v^{m})&=&\partial _{xx}(w^{m})(1-Aw^{\eta})^{m}+2\partial _x(w^{m})\partial _x((1-Aw^{\eta})^{m})+w^{m}\partial_{xx}((1-Aw^{\eta})^{m})\nonumber \\
&= &  w^{m+\beta-1} \varphi' \left[m  (1-Aw^\eta )^m - A m \eta  w^\eta (1-Aw^\eta)^{m-1} \right]\\
&& + w^{m+2\beta-2} \varphi^2 \Big[ m(m+\beta- 1)  (1-A w^\eta)^m - 2 A m^2 \eta  w^\eta (1-Aw^\eta)^{m-1}  \\
&& \hspace{2cm} - Am\eta(\beta+\eta -1) w^{\eta}(1-Aw^{\eta})^{m-1}\\
&& \hspace{2cm} - (1-m) A^2 m \eta^2  w^{2 \eta} (1-Aw^\eta)^{m-2} \Big],
\end{eqnarray*}
where again $\varphi = \frac{\tilde{u}_0 '}{\tilde{u}_0^\beta}$.

Let us first assume that $w (t,x) \leq \delta$, which is equivalent to $x \geq X_1 (t)$ for some well-chosen $X_1 (t)$. Since $\delta$ satisfies \eqref{def-delta1-ter} and \eqref{def-delta2-ter}, and using also the fact that $\varphi '$ is positive, it is straightforward to check that $\partial_{xx} (v^m) \geq 0$. Then, thanks to \eqref{def-rho-kppter},
$$\mathcal{L} v \leq (\rho - r ) w^\beta + (rA\beta - A \rho (1+\eta)) w^{\beta + \eta} \leq 0.$$
It remains to check the parabolic inequality when $\left(\frac 1{A(1+\eta)}\right)^{1/\eta}\geq 
w(t,x) \geq  \delta$, i.e. when $X (t) \leq x \leq X_1 (t)$; this is where the above choice of $T$ matters. We go back to the computation of $\partial_{xx} (v^m)$. First, we remove some positive terms and find that
\begin{eqnarray*}
\partial _{xx}(v^{m})&\geq&
  -A m \eta \varphi ' w^{m+\eta+\beta-1}(1-Aw^\eta)^{m-1}+ w^{m+2\beta +\eta -2} \Big[  - 2 A m^2 \eta \varphi^2  (1-Aw^\eta)^{m-1}  \\
&& \hspace{0.8cm} - Am\eta(\beta+\eta -1)\varphi ^2  (1-Aw^{\eta})^{m-1} - (1-m) A^2 m \eta^2 \varphi^2 w^{ \eta} (1-Aw^\eta)^{m-2} \Big].
\end{eqnarray*}
Now, using $0<m<1$ and $1 \geq 1 - A w^\eta (t,x) \geq \frac{\eta}{1+\eta} \geq \frac{1}{2}$ (recall that $\eta > \beta +2$), we obtain that
\begin{eqnarray}
- \partial _{xx}(v^{m})& \leq  &  w^{\beta } \left[2^{1-m} Am\eta \varphi ' w^{m + \eta -1}+ 2^{2-m} Am^{2}\eta \varphi ^{2}  w^{m+\eta+  \beta -2} \nonumber \right.\\
&& \hspace{0.8cm} +  2^{1-m} Am\eta(\beta+\eta -1)\varphi ^2 w^{m+\eta+ \beta  -2} \nonumber\\
& &  \hspace{0.8cm}\left. + 2^{2-m}(1-m)   A^2 m \eta^2 \varphi^2 w^{m+2 \eta+ \beta-2} \right].\nonumber
\end{eqnarray}
But the exponents $m+\eta -1$ and $ m + \eta + \beta -2$ are positive, and we also know that $w \leq \frac{1}{A^{1/ \eta}}$. Thus, using also \eqref{phi-estimter},
\begin{eqnarray}
- \partial _{xx}(v^{m})& \leq  &  w^{\beta } \left[ 2^{1-m} m\eta \varphi ' A^{-\frac{m  -1}{\eta}}+2^{2-m} m^{2}\eta \varphi ^{2}  A^{-\frac{m+  \beta -2}{\eta}}    \nonumber \right. \\
&&  \hspace{0.8cm} +  2^{1-m} m\eta(\beta+\eta -1)\varphi ^2 A^{-\frac{m+ \beta  -2}{\eta}} \nonumber\\
& &  \hspace{0.8cm} \left. + 2^{2-m}(1-m)   m \eta^2 \varphi^2 A^{-\frac{m+ \beta-2}{\eta}} \right], \nonumber\\
& \leq & (r - \rho) w^\beta. \nonumber
\end{eqnarray}
We conclude that 
$$\mathcal{L} v \leq (\rho - r ) w^\beta + (r - \rho) w^\beta + (rA\beta - A \rho (1+\eta)) w^{\beta + \eta} \leq 0.$$
Lemma~\ref{lem:sub-kppter} is proved.
 \end{proof}
 
Before applying a comparison principle, we have to show that the initial conditon $u_0$ and the subsolution $v(T,\cdot)$ of Lemma \ref{lem:sub-kppter} are ordered. The difficulty here is that $T$ may be large while $v$ is increasing in time. To circumvent this, we will prove that there exists a spatial shift $X>0$ such that
\begin{equation}\label{u0v-order-ter}
u_0 (\cdot - X) \geq v (T, \cdot).
\end{equation}
To do so, we look at the asymptotics of $v (T,x)$ as $x \to +\infty$. Recalling that $\beta >1$ and using the explicit expression for $\tilde{u}_0$ when $x$ is large, we have that
$$v(T,x) = \frac{1}{\left( \frac{x^{\alpha (\beta-1)}}{(C-\varepsilon)^{\beta-1} } - \rho (\beta -1)T\right)^{\frac{1}{\beta-1}}} \times \left( 1-A \frac{1}{\left( \frac{x^{\alpha (\beta-1)}}{(C-\varepsilon)^{\beta-1} }- \rho (\beta -1)T\right)^{\frac{\eta}{\beta-1}}} \right).$$
Thus, as $x \to +\infty$,
$$v(T,x) \sim \frac{C-\varepsilon}{x^\alpha},$$
and there exists $X'>x_0$  large enough so that
$$v(T,x) \leq \frac{C }{x^\alpha},  \quad \forall x \geq X'.$$
Recall also that, for all $x \in \mathbb{R}$,
$$v(T,x) \leq \left(\frac{1}{A(1+\eta)} \right)^{1/\eta} \frac{\eta}{1+\eta} < \kappa .$$
Now choose $X =X' - x_0$, and find that
$$ u_0 (x-X) =  \frac{C}{(x-X)^\alpha} \geq \frac{C}{x^\alpha}, \quad \forall x \geq X',$$
as well as, according to \eqref{def-kappa-kppter} and $\tilde{u}_0 = \frac{C-\varepsilon}{C} u_0$, that
$$u_0 (x - X) \geq \kappa ,\quad \forall x \leq X'.$$
Putting together the last four inequalities, we find that \eqref{u0v-order-ter} holds as announced. Hence, by comparison,
$$
u(t,x)\geq v(T+t,X+x)
$$
for all $t>0$, $x\in \R$. The proof of the lower bound on the level sets then proceeds as in subsection \ref{ss:lower}, so that we omit the details. This ends the proof of Theorem~\ref{th:reste} $(ii)$. \qed

\bibliographystyle{siam}    
\bibliography{biblio2}

 \end{document}